\documentclass[11pt,a4paper]{article}

\usepackage{epsf,epsfig,amsfonts,amsgen,amsmath,amstext,amsbsy,amsopn,amsthm
}
\usepackage{amsmath,times,mathptmx}
\usepackage{amsfonts,amsthm,amssymb}
\usepackage{amsfonts}
\usepackage{graphics}
\usepackage{latexsym,bm}
\usepackage{amsfonts,amsthm,amssymb,bbding}
\usepackage{indentfirst}
\usepackage{graphicx}
\usepackage{color}
\usepackage[colorlinks=true,anchorcolor=blue,filecolor=blue,linkcolor=blue,urlcolor=blue,citecolor=blue]{hyperref}
\usepackage{float}
\usepackage{tikz}
\evensidemargin=\oddsidemargin\topmargin=-1.5cm

\newtheorem{thm}{Theorem}[section]

\newtheorem{lem}{Lemma}[section]
\newtheorem{cor}{Corollary}[section]

\newtheorem{conj}{Conjecture}[section]
\newtheorem{claim}{Claim}[section]
\newtheorem{definition}{Definition}[section]

\addtocounter{section}{0}

\begin{document}
\title{Non-bipartite graphs without theta subgraphs\footnote{Supported by the National Natural Science Foundation of China (Nos.\,12271162,\,12326372), and Natural Science Foundation of Shanghai (Nos. 22ZR1416300 and 23JC1401500) and The Program for Professor of Special Appointment (Eastern Scholar) at Shanghai Institutions of Higher Learning (No. TP2022031).}}
\author{{\bf Longfei Fang$^{a,b}$},
{\bf Huiqiu Lin$^{a}$}\thanks{Corresponding author: huiqiulin@126.com(H. Lin)} \\
\small $^{a}$ School of Mathematics, East China University of Science and Technology, \\
\small  Shanghai 200237, China\\
\small $^{b}$ School of Mathematics and Finance, Chuzhou University, \\
\small  Chuzhou, Anhui 239012, China\\
}

\date{}
\maketitle
{\flushleft\large\bf Abstract}
A generalized theta graph $\theta(\ell_1,\ell_2,\dots,\ell_{t})$ is obtained from
internally disjoint paths of lengths $\ell_1,\ell_2,\dots,\ell_{t},$ respectively by sharing a common pair of endpoints.
Let ${\rm ex}(n,H)$ and ${\rm spex}(n,H)$
be the maximum size and the maximum spectral radius
over all $H$-free graphs of order $n$, respectively.
Bukh and Tait [Combin. Probab. Comput. 29  (2020)  495--507] proved that $\mathrm{ex}(n,\theta(\ell_1,\ell_2,\dots,\ell_{t}))\leq O_{\ell}(t^{1-1/{\ell}}n^{1+1/{\ell}})$ when $\ell_1=\cdots=\ell_{t}=\ell$.
Liu and Yang [SIAM J. Discrete Math. 37 (2023) 1237--1251] showed that $\mathrm{ex}(n,\theta(\ell_1,\ell_2,\dots,\ell_{t}))=O(n^{1+2/{(\ell_1+\ell_2)}})$,
where $\ell_1,\dots,\ell_t$ have the same parity, $\ell_1\leq \cdots \leq \ell_t$ and $\ell_2\geq 2$.
%Zhai and Lin [J. Graph Theory 102 (2023) 502--520] proved that $T_{n,2}$ is the unique extremal graph with respect to $\mathrm{spex}(n,\theta(1,2,r))$ when $n\geq 10r$ for odd $r\geq 3$ and $n\geq 7r$ for even $r$.
Fix a color-critical graph $H$ with $\chi(H)=r+1\geq 3$.
Simonovits' chromatic critical edge theorem and Nikiforov's spectral chromatic critical edge theorem imply that $T_{n,r}$ is the extremal graph with respect to ${\rm ex}(n,H)$ and ${\rm spex}(n,H)$
for $n$ sufficiently large, respectively.
Since $T_{n,r}$ is $r$-partite,
it is interesting to study the Tur\'{a}n number and the spectral Tur\'{a}n number of a color-critical graph $H$ in non-$r$-partite graphs.
Denote by ${\rm EX}_{r+1}(n,H)$ (resp. ${\rm SPEX}_{r+1}(n,H)$) the family of $n$-vertex $H$-free non-$r$-partite graphs with the maximum size (resp. spectral radius).
Brouwer showed that any graph in $\mathrm{EX}_{r+1}(n,K_{r+1})$ is of size $e(T_{n,r})-\lfloor\frac{n}{r}\rfloor+1$ for $n\geq 2r+1$.
Lin, Ning and Wu [Combin. Probab. Comput. 30 (2) (2021) 258--270], and
Li and Peng [SIAM J. Discrete Math. 37 (2023)  2462--2485] characterized the unique graph in $\mathrm{SPEX}_{r+1}(n,K_{r+1})$ for $r\geq 2$.
Particularly, the unique graph is of size $e(T_{n,r})-\lfloor\frac{n}{r}\rfloor+1$.
Thus $\mathrm{SPEX}_{r+1}(n,K_{r+1})\subseteq \mathrm{EX}_{r+1}(n,K_{r+1})$.
It is natural to conjecture that ${\rm SPEX}_{r+1}(n,H)\subseteq {\rm EX}_{r+1}(n,H)$
for arbitrary color-critical graph $H$ with $\chi(H)=r+1\geq 3$.
Fix $q,r\geq 2$ with even $q$,
 $\theta(1,q,r)$ is a color-critical graph with chromatic number three.
In this paper, we prove that $\mathrm{SPEX}_{3}(n,\theta(1,q,r))\subseteq \mathrm{EX}_{3}(n,\theta(1,q,r))$  for sufficiently large $n$.
Furthermore, we determine all the graphs in $\mathrm{SPEX}_{3}(n,\theta(1,q,r))$ and $\mathrm{EX}_{3}(n,\theta(1,q,r))$, respectively.
%Our results generalized some previous result on odd cycles.

\begin{flushleft}
\end{flushleft}

\section{Introduction}

Given a graph $H$.
We say a graph \textit{$H$-free}
if it does not contain $H$ as a subgraph.
The \emph{Tur\'{a}n number} of $H$, denoted by $\mathrm{ex}(n,H)$,
is the maximum number of edges in an $n$-vertex $H$-free graph.
Let $\mathrm{EX}(n,H)$ denote the family of $n$-vertex $H$-free graphs with the maximum number of edges.
As one of the earliest results in extremal graph theory, Tur\'{a}n's theorem \cite{Turan} states that
${\rm EX}(n,K_{r+1})=\{T_{n,r}\}$,
where the Tur\'{a}n graph $T_{n,r}$ denotes the complete $n$-vertex $r$-partite graph with part sizes as
equal as possible.
A generalized theta graph $\theta(\ell_1,\ell_2,\dots,\ell_{t})$ is obtained from
internally disjoint paths of lengths $\ell_1,\ell_2,\dots,\ell_{t},$ respectively by sharing a common pair of endpoints.
Specifically, $\theta(\ell,\ell)=C_{2\ell}$ and $\theta(\ell,\ell+1)=C_{2\ell+1}$.
Thus, the study of determining  $\mathrm{ex}(n,\theta(\ell_1,\ell_2,\dots,\ell_{t}))$ generalizes the problem of determining $\mathrm{ex}(n,C_{k})$.
F\"{u}redi and Gunderson \cite{FG-2015} characterized all the graphs in $\mathrm{EX}(n,C_{2\ell+1})$ for $\ell\geq 2$.
Set $\theta_{\ell,t}=\theta(\ell_1,\ell_2,\dots,\ell_{t})$ when $\ell_1=\cdots=\ell_{t}=\ell$.
Already in the 80s, Faudree and Simonovits \cite{Faudree1983} established $\mathrm{ex}(n,\theta_{\ell,t})=O_{\ell,t}(n^{1+1/{\ell}})$.
Bukh and Tait \cite{Bukh2020} improved the result and obtained $\mathrm{ex}(n,\theta_{\ell,t})\leq c_{\ell}t^{1-1/{\ell}}n^{1+1/{\ell}}$ for some constant $c_{\ell}$ depending on $\ell$.
Recently, Liu and Yang \cite{Liu2023} showed that $\mathrm{ex}(n,\theta(\ell_1,\ell_2,\dots,\ell_{t}))=O(n^{1+2/{(\ell_1+\ell_2)}})$,
where $\ell_1,\dots,\ell_t$ have the same parity, $\ell_1\leq \cdots \leq \ell_t$ and $\ell_2\geq 2$.
When restrain $t=3$, we obtain a theta graph $\theta(\ell_1,\ell_2,\ell_3)$.
In 2019, Verstra\"{e}te and Williford \cite{Verstraete2019} constructed a $\theta(4,4,4)$-free graph with $(\frac12-o(1))n^{5/4}$ edges.

%F\"{u}redi and Gunderson \cite{Furedi2015} characterized $\mathrm{EX}(n,C_{2\ell+1})$ for all $\ell$ and $n$.

Denote by $A(G)$ and $\rho(G)$ the \emph{adjacency matrix} and the \emph{spectral radius} of a graph $G$, respectively.
%The \emph{spectral extremal value} of a given graph $H$, written by $\mathrm{spex}(n,H)$,
%is the maximum spectral radius in an $n$-vertex $H$-free graph.
Let $\mathrm{SPEX}(n,H)$ denote the family of $n$-vertex $H$-free graphs with the maximum spectral radius.
%That is, $\mathrm{SPEX}(n,\theta(1,2,r))=\mathrm{EX}(n,\theta(1,2,r))=\{T_{n,2}\}$ for large enough $n$.
In 2022, Cioab\u{a}, Desai and Tait \cite{Cioaba2022} conjectured that if $H$ is a graph satisfying that every graph in ${\rm EX}(n,H)$ is obtained from Tur\'{a}n graph by adding $O(1)$ edges,
then ${\rm SPEX}(n,H)\subseteq {\rm EX}(n,H)$ for sufficiently large $n$.
This conjecture has been confirmed for some special cases of $H$;
such as complete graphs \cite{Guiduli-1996,Nikiforov-2007},
friendship graphs \cite{Cioaba-2020,ZHAI2022},
intersecting cliques \cite{Desai-2022},
and intersecting odd cycles \cite{Li-2022}.
Recently, Wang, Kang and Xue \cite{Wang2023} completely solved the conjecture and gave a stronger result.

\begin{thm}\label{THM1.1}\emph{(\cite{Wang2023})}
Let $r\geq 2$, $n$ be sufficiently large, and $H$ be a graph with ${\rm ex}(n,H)=e(T_{n,r})+O(1)$.
Then ${\rm SPEX}(n,H)\subseteq {\rm EX}(n,H)$.
\end{thm}

A graph $H$ is called \emph{color-critical} if
 there exists an edge $e\in E(H)$ such that $\chi(H-\{e\})<\chi(H)$,
where $\chi(H)$ denotes the \emph{chromatic number} of $H$.
Zhai, Fang and Shu \cite{Zhai2021} showed that
$\theta(p,q,r)$ is a color-critical graph with chromatic number three for any $p,q,r$ with different parities.
Moreover, they determined $\mathrm{EX}(n,\theta(p,q,r))=\{T_{n,2}\}$ for
$n\geq 9(p+q+r-1)^2-3(p+q+r-1)$.
Subsequently, when $n\geq 10(r-1)$ for even $r$ and $n\geq 7(r-1)$ for odd $r\geq 3$, $\mathrm{SPEX}(n,\theta(1,2,r))=\{T_{n,2}\}$ was given by Zhai and Lin \cite{Zhai2023},
as a direct corollary, $G$ contains a consecutive cycle of length in $[3,n/7]$ if $\rho(G)>\rho(T_{n,2})$.
Fix an arbitrary color-critical graph $H$ with $\chi(H)=r+1\geq 3$.
Simonovits \cite{Simonovits1968} proved that
 there exists an integer $n_0(H)$ such that $\mathrm{EX}(n,H)=\{T_{n,r}\}$ when $n\geq n_0(H)$,
which is known as the \emph{chromatic critical edge theorem}.
Nikiforov's result (see \cite[Theorem 2]{Nikiforov-2009-2}) implies that there exists an integer $n_0(H)\geq e^{|V(H)|r^{(2r+9)(r+1)}}$ such that $\mathrm{SPEX}(n,H)=\{T_{n,r}\}$ when $n\geq n_0(H)$,
which is known as the \emph{spectral chromatic critical edge theorem}.
That is, for large enough $n$, $\mathrm{SPEX}(n,H)=\mathrm{EX}(n,H)=\{T_{n,r}\}$.
Notice that $T_{n,r}$ is $r$-partite.
Based on these observations, we consider Tur\'{a}n-type problems and spectral Tur\'{a}n-type problems
for a color-critical graph $H$ in non-$r$-partite graphs.
Let ${\rm EX}_{r+1}(n,H)$ (resp. ${\rm SPEX}_{r+1}(n,H)$) denote the family of $n$-vertex $H$-free non-$r$-partite graphs with the maximum size (resp. spectral radius).
The aforementioned  maximum size is denoted by ${\rm ex}_{r+1}(n,H)$.
Erd\H{o}s showed that $\mathrm{ex}_{3}(n,K_{3})=\lfloor\frac{(n-1)^2}{4}\rfloor+1$
(see \cite[p. 306]{Bondy2008}).
Brouwer \cite{Brouwer1981} determined $\mathrm{ex}_{r+1}(n,K_{r+1})=e(T_{n,r})-\lfloor\frac{n}{r}\rfloor+1$ for $n\geq 2r+1$.
As for spectral extremal result, Lin, Ning and Wu \cite{Lin-N2021} obtained that $\mathrm{SPEX}_{3}(n,K_{3})=\{SK_{\lceil\frac{n-1}{2}\rceil,\lfloor\frac{n-1}{2}\rfloor}\}$.
Recently, Li and Peng \cite{Li-P2023} characterized the unique graph in $\mathrm{SPEX}_{r+1}(n,K_{r+1})$,
which has exactly $e(T_{n,r})-\lfloor\frac{n}{r}\rfloor+1$ edges.
Thus ${\rm SPEX}_{r+1}(n,K_{r+1})\subseteq {\rm EX}_{r+1}(n,K_{r+1})$ for $n\geq 2r+1$.
A result of Ren, Wang, Wang, and Yang implies that
$\mathrm{ex}_{3}(n,C_{2\ell+1})=\lfloor\frac{(n-2)^2}{4}\rfloor+3$ for $\ell\geq 2$ and $n\geq 318\ell^2$ (see \cite[Theorem 1.3]{Ren2024}).
In \cite{Guo20212, Zhang2023}, they obtained that $\mathrm{SPEX}_{3}(n,C_{2\ell+1})=\{K_{\lceil\frac{n-2}{2}\rceil,\lfloor\frac{n-2}{2}\rfloor}\circ K_3\}$ for $\ell\geq 2$ and sufficiently large $n$.
Since $e(K_{\lceil\frac{n-2}{2}\rceil,\lfloor\frac{n-2}{2}\rfloor}\circ K_3)=\lfloor\frac{(n-2)^2}{4}\rfloor+3$,
we can see that $\mathrm{SPEX}_{3}(n,C_{2\ell+1})\subseteq \mathrm{EX}_{3}(n,C_{2\ell+1})$ for $\ell\geq 2$ and sufficiently large $n$.
It is not hard to check that if $q,r\geq 2$ are integers with even $q$,
then $\theta(1,q,r)$ is a color-critical graph with chromatic number three.
Bataineh, Jaradat and Al-Shboul \cite{Bataineh2016} obtained $\mathrm{ex}_{3}(n,\theta(1,2,3))=\lfloor\frac{(n-1)^2}{4}\rfloor+1$ for $n\geq 9$.
Recently, Li, Sun and Wei \cite{Li2023} proved that $\mathrm{EX}_{3}(n,\theta(1,2,4))=\{K_{\lceil\frac{n-1}{2}\rceil,\lfloor\frac{n-1}{2}\rfloor}\bullet K_3\}$ for $n\geq 137$,
$\mathrm{SPEX}_{3}(n,\theta(1,2,3))=\{SK_{\lfloor\frac{n-1}{2}\rfloor,\lceil\frac{n-1}{2}\rceil}\}$ for $n\geq 20$, and $\mathrm{SPEX}_{3}(n,\theta(1,2,4))=\{K_{\lceil\frac{n-1}{2}\rceil,\lfloor\frac{n-1}{2}\rfloor}\bullet K_3\}$ for $n\geq 21$.
Since $e(SK_{\lfloor\frac{n-1}{2}\rfloor,\lceil\frac{n-1}{2}\rceil})=\lfloor\frac{(n-1)^2}{4}\rfloor+1$,
their results imply that ${\rm SPEX}_{r+1}(n,\theta(1,2,r))\subseteq {\rm EX}_{r+1}(n,\theta(1,2,r))$  for $r\in \{3,4\}$ and large enough $n$.
More relevant results can be seen in \cite{Amin2013,Khadziivanov1979,Tyomkyn2015,Zhang2023}.
Inspired by the above results, it is natural to consider the following conjecture.

\begin{conj}\label{CONJ1.1}
Let $H$ be an arbitrary color-critical graph with $\chi(H)=r+1\geq 3$.
For sufficiently large $n$, ${\rm SPEX}_{r+1}(n,H)\subseteq {\rm EX}_{r+1}(n,H)$.
\end{conj}

To state our main results,
we first introduce several graphs and graph families.
Let $P_n,C_n,K_n$ and $K_{a,n-a}$ denote a path, a cycle, a complete graph and a complete bipartite graph of order $n$, respectively.
Denote by $K_{a,b}\circ K_3$ the graph obtained by identifying a vertex of $K_{a,b}$
belonging to the part of size $b$ and a vertex of $K_3$.
Denote by $K_{a,b}\bullet K_3$ the graph obtained by identifying an edge of $K_{a,b}$
with an edge of $K_3$.
Denote by $SK_{a,b}$  the graph obtained from $K_{a,b}$ by subdividing an edge.

\begin{figure}%[!h]
\centering
\begin{tikzpicture}[scale=0.75, x=1.00mm, y=1.00mm, inner xsep=0pt, inner ysep=0pt, outer xsep=0pt, outer ysep=0pt]
\definecolor{L}{rgb}{0,0,0}
\definecolor{F}{rgb}{0,0,0}

\node[draw,dashed, line width=0.3mm, minimum width=45mm, minimum height=9mm] (rect1) at (10,0) {};
% 绘制一个长方形
\draw(8,8) node[anchor=base west]{\fontsize{12.23}{17.07}\selectfont $Y$};

\node[draw, line width=0.3mm, minimum width=20mm, minimum height=6mm] (rect11) at (-5,0) {};
% 绘制一个长方形
\draw(-15.5,-1.5) node[anchor=base west]{\fontsize{10.23}{17.07}\selectfont $Y\setminus\{u_1,u_2\}$};

\node[circle,fill=F,draw,inner sep=0pt,minimum size=2mm] (u1) at (20.00,00.00) {};
\draw(19,2) node[anchor=base west]{\fontsize{10.23}{17.07}\selectfont $u_2$};

\node[circle,fill=F,draw,inner sep=0pt,minimum size=2mm] (u2) at (30.00,00.00) {};
\draw(29,2) node[anchor=base west]{\fontsize{10.23}{17.07}\selectfont $u_1$};

\node[draw,dashed, line width=0.3mm, minimum width=45mm, minimum height=9mm] (rect2) at (10,-30) {};
% 绘制一个长方形
\draw(8,-42) node[anchor=base west]{\fontsize{12.23}{17.07}\selectfont $X$};

\node[draw, line width=0.3mm, minimum width=20mm, minimum height=6mm] (rect21) at (-5,-30) {};
% 绘制一个长方形
\draw(-11,-31.5) node[anchor=base west]{\fontsize{10.23}{17.07}\selectfont $X\setminus X_1$};
\node[draw, line width=0.3mm, minimum width=20mm, minimum height=6mm] (rect22) at (25,-30) {}; % 绘制一个长方形
\draw(23,-31.5) node[anchor=base west]{\fontsize{10.23}{17.07}\selectfont $X_1$};

\definecolor{L}{rgb}{0,0,0}
\path[line width=0.3mm, draw=L] (u1) -- (u2);
\path[line width=0.3mm, draw=L] (u1) -- (rect21);
\path[line width=0.3mm, draw=L] (u2) -- (rect22);
\path[line width=0.3mm, draw=L] (rect11) -- (rect21);
\path[line width=0.3mm, draw=L] (rect11) -- (rect22);

%右边图

\node[draw,dashed, line width=0.3mm, minimum width=45mm, minimum height=9mm] (rect1b) at (90,0) {};
% 绘制一个长方形
\draw(88,8) node[anchor=base west]{\fontsize{12.23}{17.07}\selectfont $Y$};

\node[draw, line width=0.3mm, minimum width=20mm, minimum height=6mm] (rect12b) at (75,0) {};
% 绘制一个长方形
\draw(70,-1.5) node[anchor=base west]{\fontsize{10.23}{17.07}\selectfont $Y\setminus Y_1$};

\node[draw, line width=0.3mm, minimum width=20mm, minimum height=6mm] (rect11b) at (105,0) {};
% 绘制一个长方形
\draw(102,-1.5) node[anchor=base west]{\fontsize{10.23}{17.07}\selectfont $Y_1$};

\node[draw,dashed, line width=0.3mm, minimum width=45mm, minimum height=9mm] (rect2b) at (90,-30) {};
% 绘制一个长方形
\draw(88,-42) node[anchor=base west]{\fontsize{12.23}{17.07}\selectfont $X$};

\node[draw, line width=0.3mm, minimum width=20mm, minimum height=6mm] (rect22b) at (75,-30) {};
% 绘制一个长方形
\draw(70,-31.5) node[anchor=base west]{\fontsize{10.23}{17.07}\selectfont $X\setminus X_1$};

\node[draw, line width=0.3mm, minimum width=20mm, minimum height=6mm] (rect21b) at (105,-30) {};
% 绘制一个长方形
\draw(102,-31.5) node[anchor=base west]{\fontsize{10.23}{17.07}\selectfont $X_1$};

\node[circle,fill=F,draw,inner sep=0pt,minimum size=2mm] (w1) at (135.00,00.00) {};
\draw(134,2) node[anchor=base west]{\fontsize{10.23}{17.07}\selectfont $w_1$};

\node[circle,fill=F,draw,inner sep=0pt,minimum size=2mm] (w2) at (135.00,-30.00) {};
\draw(134,-35) node[anchor=base west]{\fontsize{10.23}{17.07}\selectfont $w_2$};

\node[circle,fill=F,draw,inner sep=0pt,minimum size=2mm] (w3) at (150.00,-15.00) {};
\draw(152,-16) node[anchor=base west]{\fontsize{10.23}{17.07}\selectfont $w_3$};

\definecolor{L}{rgb}{0,0,0}
\path[line width=0.3mm, draw=L] (rect11b) -- (rect21b);
\path[line width=0.3mm, draw=L] (rect11b) -- (rect22b);
\path[line width=0.3mm, draw=L] (rect12b) -- (rect21b);
\path[line width=0.3mm, draw=L] (rect12b) -- (rect22b);

\path[line width=0.3mm, draw=L] (w1) -- (w2);
\path[line width=0.3mm, draw=L] (w1) -- (w3);
\path[line width=0.3mm, draw=L] (w2) -- (w3);

\path[line width=0.3mm, draw=L] (w1) -- (rect22b);
\path[line width=0.3mm, draw=L] (w2) -- (rect12b);
\path[line width=0.3mm, draw=L] (w3) -- (rect11b);
\path[line width=0.3mm, draw=L] (w3) -- (rect21b);

\draw(5,-50) node[anchor=base west]{\fontsize{14.23}{17.07}\selectfont $(a)$};
\draw(95,-50) node[anchor=base west]{\fontsize{14.23}{17.07}\selectfont $(b)$};

\end{tikzpicture}
\caption{$(a)$ Structure of graphs in $\mathcal{H}(n)$, and $(b)$ structure of graphs in $\mathcal{G}(n)$.}{\label{fig-1.1}}
\end{figure}
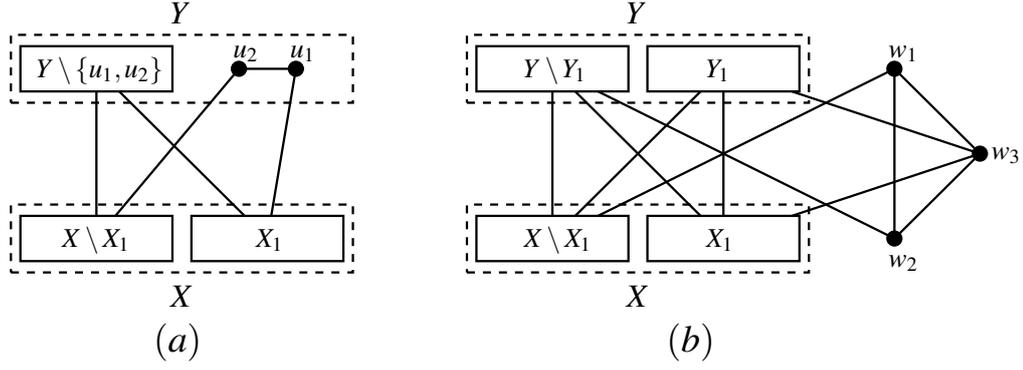

\begin{definition}\label{DEF1.1}
Let $X$ and $Y$ be color parts of $K_{\lceil\frac{n}{2}\rceil+b,\lfloor\frac{n}{2}\rfloor-b}$ with $|Y|=\lceil\frac{n}{2}\rceil+b$ and $u_1,u_2\in Y$,
where $b=0$ when $n$ is odd and $b\in \{0,1\}$ otherwise.
Define $\mathcal{H}(n)$ as the family of graphs obtained from $K_{\lceil\frac{n}{2}\rceil+b,\lfloor\frac{n}{2}\rfloor-b}$ by adding an edge $u_1u_2$,
and then deleting edges from $u_1$ to $X\setminus X_1$,
and from $u_2$ to $X_1$,
where $X_1$ is a non-empty proper subset of $X$
(see Figure \ref{fig-1.1} ($a$)).
\end{definition}

\begin{definition}\label{DEF1.2}
Let $X$ and $Y$ be color parts of $T_{n-3,2}$, $X_1\subseteq X$ and $Y_1\subseteq Y$.
Let $C_3$ be a triangle with vertex set $\{w_1,w_2,w_3\}$.
Define $\mathcal{G}(n)$ as the family of graphs obtained from $T_{n-3,2}$ and $C_3$
by joining $w_1$ to all vertices in $X\setminus X_1$,
 $w_2$ to all vertices in $Y\setminus Y_1$,
and $w_3$ to all vertices in $X_1\cup Y_1$, where  $|X_1|\cdot|Y_1|\leq1$
(see Figure \ref{fig-1.1} ($b$)).
\end{definition}

In this paper,
%we prove that Conjecture \ref{CONJ1.1} holds for $\theta(1,q,r)$ for $q,r\geq 2$ with even $q$.
%Furthermore,
we characterize all the graphs in $\mathrm{EX}_{3}(n,\theta(1,q,r))$ and $\mathrm{SPEX}_{3}(n,\theta(1,q,r))$ for sufficiently large $n$, respectively.

\begin{thm}\label{THM1.2}
Let $q,r\geq 2$ be integers with even $q$, and $n$ be sufficiently large. Then
 $${\rm EX}_{3}(n,\theta(1,q,r))=\left\{
                                       \begin{array}{ll}
                                        \mathcal{G}(n)  & \hbox{if $q=r=2$,} \\
                                        \{K_{\lceil\frac{n-1}{2}\rceil,\lfloor\frac{n-1}{2}\rfloor}\bullet K_3\} & \hbox{if $r\geq 4$ is even,}\\
                                         \mathcal{H}(n)  & \hbox{if $q=2$ and $r$ is odd,}\\
                                        \{K_{\lceil\frac{n-2}{2}\rceil,\lfloor\frac{n-2}{2}\rfloor}\circ K_3,K_{\lfloor\frac{n-2}{2}\rfloor,\lceil\frac{n-2}{2}\rceil}\circ K_3\} & \hbox{if $q\geq 4$ and $r$ is odd.}
                                       \end{array}
                                     \right.
$$

\end{thm}

\begin{thm}\label{THM1.3}
Let $q,r\geq 2$ be integers with even $q$, and $n$ be sufficiently large. Then
 $${\rm SPEX}_{3}(n,\theta(1,q,r))=\left\{
                                       \begin{array}{ll}
                                        \{K_{\lceil\frac{n-1}{2}\rceil,\lfloor\frac{n-1}{2}\rfloor}\bullet K_3\}  & \hbox{if $r$ is even,} \\
                                        \{SK_{\lceil\frac{n-1}{2}\rceil,\lfloor\frac{n-1}{2}\rfloor}\}  & \hbox{if $r$ is odd and $q=2$,}\\
                                         \{K_{\lceil\frac{n-2}{2}\rceil,\lfloor\frac{n-2}{2}\rfloor}\circ K_3\}  & \hbox{if $r$ is odd and $q\geq 4$.}
                                       \end{array}
                                     \right.
$$
\end{thm}

Setting $|X_1|=1$ in Definition \ref{DEF1.1},
we can see that $SK_{\lceil\frac{n-1}{2}\rceil,\lfloor\frac{n-1}{2}\rfloor}\in \mathcal{H}(n)$.
Setting $X_1=Y_1=\varnothing$ in Definition \ref{DEF1.2},
we can see that $K_{\lceil\frac{n-1}{2}\rceil,\lfloor\frac{n-1}{2}\rfloor}\bullet K_3\in \mathcal{G}(n)$.
Combining these with Theorems \ref{THM1.2} and \ref{THM1.3},
we have the following result.

\begin{cor}\label{COR1.1}
Let $q,r\geq 2$ be integers with even $q$, and $n$ be sufficiently large.
Then Conjecture \ref{CONJ1.1} holds for $\theta(1,q,r)$.
\end{cor}

Fix an integer $\ell\geq 2$.
Since $C_{2\ell+1}\subseteq \theta(1,2\ell,2\ell+1)$,
it follows that
a $C_{2\ell+1}$-free graph is also a $\theta(1,2\ell,2\ell+1)$-free graph.
Clearly, $K_{\lceil\frac{n-2}{2}\rceil,\lfloor\frac{n-2}{2}\rfloor}\circ K_3$ is $C_{2\ell+1}$-free.
Combining these with Theorems \ref{THM1.2} and \ref{THM1.3},
we can see that $K_{\lceil\frac{n-2}{2}\rceil,\lfloor\frac{n-2}{2}\rfloor}\circ K_3\in {\rm EX}_{3}(n,C_{2\ell+1})$ and ${\rm SPEX}_{3}(n,C_{2\ell+1})=\{K_{\lceil\frac{n-2}{2}\rceil,\lfloor\frac{n-2}{2}\rfloor}\circ K_3\}$ for $n$ sufficiently large.
Thus, we have the following result.

\begin{cor}\label{COR1.2}
\emph{(\cite{Guo20212,Ren2024,Zhang2023})}
Let $\ell\geq 2$ be an integer and $n$ be sufficiently large.
Then  Conjecture \ref{CONJ1.1} holds for $C_{2\ell+1}$.
\end{cor}

\section{Proof of Theorem \ref{THM1.2}}\label{section2}

Given a simple graph $G$,
we use $V(G)$ to denote the vertex set, $E(G)$ the edge set, and $e(G)$ the number of edges, respectively.
For a vertex $v\in V(G)$,
we denote by $N_G(v)$ its neighborhood
and set $d_G(v)=|N_G(u)|$.
For an edge subset $M\subseteq E(G)$,
we define $G-M=(V(G),E(G)\setminus M)$.
Given  a vertex $v\in V(G)$ and  two disjoint vertex subsets $S$ and $T$.
Set $N_S(v)=N_G(v)\cap S$ and $d_S(v)=|N_S(v)|$.
Let $G[S]$ be the subgraph induced by $S$ and set $G-S=G[V(G)\setminus S]$.
Let $G[S,T]$ be the bipartite subgraph on the vertex set $S\cup T$
which consists of all edges with one endpoint in $S$ and the other in $T$.
When there is no ambiguity, we write $e(S)$ and $e(S,T)$ respectively for $e(G[S])$ and $e(G[S,T])$.

%If there is no danger of ambiguity, we drop the notation $H$.

In this section, we first list some lemmas which will be used in the proof of Theorem \ref{THM1.2}.

\begin{lem} \label{LEM2.1}\emph{(\cite{Simonovits1968})}
Let $r\geq 2$ and $H$ be a color-critical graph with $\chi(H)=r+1$.
For sufficiently large $n$, ${\rm EX}(n,H)=\{T_{n,r}\}$.
\end{lem}

\begin{lem}\label{LEM2.2}\emph{(\cite{Caccetta2002})}
Let $G$ be a non-bipartite graph on $n$ vertices with no odd cycle of length at most $2k+1$.
Then $e(G)\leq \lfloor\frac{(n-2k+1)^2}{4}\rfloor+2k-1$.
Moreover, ${\rm EX}_{3}(n,C_3)=\mathcal{H}(n)$.
\end{lem}

The classical Stability Lemma was given by Erd\H{o}s \cite{Erdos-1967,Erdos-1968}
and Simonovits \cite{Simonovits1968},
which plays a very important role in extremal graph theory.

\begin{lem} \label{LEM2.3}
\emph{(\cite{Erdos-1967,Erdos-1968,Simonovits1968})}
Let $H$ be a graph with $\chi(H)=r+1\geq 3$.
For every $\varepsilon>0$,
there exist a constant $\delta>0$ and an integer $n_0$ such that
if $G$ is an $H$-free graph on $n\geq n_0$ vertices with $e(G)\geq (\frac{r-1}{r}-\delta)\frac{n^2}{2}$,
then $G$ can be obtained from $T_{n,r}$ by adding and deleting at most $\varepsilon n^2$ edges.
\end{lem}

\begin{definition}\label{DEF2.1}
Fix integers $q,r\geq 2$ with even $q$.
A graph $G$ is called $\theta(1,q,r)$-good if
$G$ is an $n$-vertex $F$-free non-bipartite graph and is obtained from $T_{n,2}$
by adding and deleting at most $\varepsilon n^2$ edges, where
\begin{equation}\label{EQU-1}
k=|V(F)|~\text{and}~\max\{40k\varepsilon^{\frac12}, 400\varepsilon^{\frac12}\}<1.
\end{equation}
\end{definition}

Let $G$ and $\theta(1,q,r)$ be defined as in Definition \ref{DEF2.1}.
In the following, we shall prove seven claims
for an arbitrary  $\theta(1,q,r)$-good graph $G$ of sufficiently large order $n$.

\begin{claim}\label{CLA2.1}
For sufficiently large $n$, $e(G)\geq (\frac{1}{4}-\varepsilon){n^2}$.
$G$ admits a partition $V(G)=V_1\cup V_2$ such that $e(V_1,V_2)$ is maximum,
$e(V_1)+e(V_2)\leq \varepsilon n^2$ and $\big||V_i|-\frac{n}{2}\big|\leq 2\varepsilon^{\frac12} n$ for each $i\in \{1,2\}.$
\end{claim}

\begin{proof}

By Definition \ref{DEF2.1}, we have $e(G)\geq (\frac{1}{4}-\varepsilon){n^2}$,
and there exists a partition $V(G)=U_1\cup U_2$ such that
$\lfloor\frac{n}{2}\rfloor\le |U_1|\leq|U_2|\leq \lceil\frac{n}{2}\rceil$ and $e(U_1)+e(U_2)\leq \varepsilon n^2$.
We now select a new partition $V(G)=V_1\cup V_2$
such that $e(V_1,V_2)$ is maximum.
Then $e(V_1)+e(V_2)$ is minimum, and so
\begin{center}
$e(V_1)+e(V_2)\leq e(U_1)+e(U_2)\leq \varepsilon n^2.$
\end{center}
On the other hand, assume that $|V_1|=\frac{n}{2}+\alpha$ for some $\alpha$, then
$|V_1||V_2|=\frac{n^2}{4}-\alpha^2$.
Thus,
\begin{center}
  $e(G)= e(V_1,V_2)+e(V_1)+e(V_2)\leq \frac{n^2}{4}-\alpha^2+\varepsilon n^2.$
\end{center}
Combining $e(G)\geq \frac{n^2}{4}-\varepsilon n^2$,
we get $\alpha^2\leq 2\varepsilon n^2$,
and so $||V_2|-\frac{n}{2}|=||V_1|-\frac{n}{2}|=|\alpha|\leq 2\varepsilon^{\frac12} n$.
\end{proof}

In the following, we shall define two vertex subsets $S$ and $W$ of $G$.

\begin{claim}\label{CLA2.2}
Let $S=\{v\in V(G)~|~d_G(v)\leq\big(\frac12-4\varepsilon^{\frac12}\big)n\}.$
Then $|S|\leq \varepsilon^{\frac12} n$.
\end{claim}

\begin{proof}
Suppose to the contrary that $|S|>\varepsilon^{\frac12} n$.
Then there exists a subset $S'\subseteq S$ with $|S'|=\lfloor \varepsilon^{\frac12} n\rfloor$.
Combining these with Claim \ref{CLA2.1}  that $e(G)\geq  (\frac14-\varepsilon\big)n^2$,  we get
\begin{equation}\label{EQU-2}
e(G-S')\geq e(G)-\sum_{v\in S'}d_G(v)
\geq \Big(\frac{1}{4}-\varepsilon\Big)n^2-\varepsilon^{\frac12} n\Big(\frac{1}{2}-4\varepsilon^{\frac12}\Big)n
=\frac{1}{4}\big(1-2\varepsilon^{\frac12}+12\varepsilon\big)n^2.
\end{equation}
Set $n'=|V(G-S')|=n-\lfloor\varepsilon^{\frac12}  n\rfloor$.
Then $(1-\varepsilon^{\frac12})n\leq n'< (1-\varepsilon^{\frac12})n+1$.
By sufficiently large $n$ and the chromatic critical edge theorem,
we have
$$e(G-S')\leq e(T_{n',2})\leq \frac{n'^2}{4}<\frac{1}{4}(1-2\varepsilon^{\frac12}+2\varepsilon)n^2,$$
which contradicts \eqref{EQU-2}.
Thus, $|S|\leq \varepsilon^{\frac12} n$.
\end{proof}

\begin{claim}\label{CLA2.3}
Let $W=W_1\cup W_2$, where $W_i=\{v\in V_i~|~d_{V_i}(v)\geq 4\varepsilon^{\frac12} n\}$ for $i\in \{1,2\}$.
Then $|W|\leq \frac12 \varepsilon^{\frac12} n$.
\end{claim}

\begin{proof}
For $i\in \{1,2\}$,
\begin{center}
  $2e(V_i)=\sum\limits_{v\in V_i}d_{V_i}(v)\ge
\sum\limits_{v\in W_i}d_{V_i}(v)\ge |W_i|\cdot 4\varepsilon^{\frac12} n.$
\end{center}
Combining this with Claim \ref{CLA2.1} gives
\begin{center}
  $\varepsilon n^2\geq e(V_1)+e(V_2)\geq (|W_1|+|W_2|) 2\varepsilon^{\frac12} n=2\varepsilon^{\frac12}|W| n.$
\end{center}
This yields that $|W|\leq \frac{1}{2}\varepsilon^{\frac12} n$.
\end{proof}

For every $i\in \{1,2\}$, denote by $\overline{V}_i=V_i\setminus (W\cup S)$.
For ease of calculation, we present the following two claims.

\begin{claim}\label{CLA2.4}
Let $i,j$ be integers with  $i\in\{1,2\}$ and $1\leq j\leq k$.
Let $u_0\in V(G)$ with $d_{V_{3-i}}(u_0)\geq 12\varepsilon^{\frac12} n$.\\
(i) If $u_1\in \overline{V}_i\setminus \{u_0\}$, then there exist at least $\frac12\varepsilon^{\frac12}n$ vertices in $\overline{V}_{3-i}$ adjacent to $u_0$ and $u_1$.\\
(ii) If $\{u_1,\dots,u_j\}\subseteq \overline{V}_i$,
then there exist at least $\varepsilon^{\frac12}n$ vertices in $\overline{V}_{3-i}$ adjacent to $u_1,\dots,u_j$.
\end{claim}

\begin{proof}
For any $u\in \overline{V}_i$,
by the definitions of $S$ and $W$,
we have $d_G(u)>(\frac{1}{2}-4\varepsilon^{\frac12})n$ and $d_{V_{i}}(u)<4\varepsilon^{\frac12} n$.
Then,
\begin{align}\label{EQU-3}
d_{V_{3-i}}(u)=d_G(u)-d_{V_{i}}(u)>\Big(\frac{1}{2}-8\varepsilon^{\frac12}\Big)n.
\end{align}

($i$) In view of $d_{V_{3-i}}(u_0)\geq 12\varepsilon^{\frac12} n$, \eqref{EQU-3}
and Claims \ref{CLA2.1}-\ref{CLA2.3}, we have
\begin{align*}
|\big(N_{V_{3-i}}(u_0)\cap N_{V_{3-i}}(u_1)\big)\setminus(W\cup S)|
&\geq |N_{V_{3-i}}(u_0)|+ |N_{V_{3-i}}(u_1)|-|V_{3-i}|-|W|-|S|\nonumber\\
&\geq 12\varepsilon^{\frac12}n+\Big(\frac{1}{2}-8\varepsilon^{\frac12}\Big)n
-\Big(\frac12 +2\varepsilon^{\frac12} \Big)n-\frac{3}{2}\varepsilon^{\frac12} n
= \frac12\varepsilon^{\frac12}n.
\end{align*}
Then, there exist at least $\frac12\varepsilon^{\frac12}n$ vertices in $\overline{V}_{3-i}$ adjacent to $u_0$ and $u_1$.

($ii$) Combining Claims \ref{CLA2.1}-\ref{CLA2.3} and \eqref{EQU-3}, we have
\begin{align*}
|\big(\bigcap_{s=1}^{j}N_{V_{3-i}}(u_s)\big)\setminus(W\cup S)|
&\geq \sum_{s=1}^{j}d_{V_{3-i}}(u_s)-(j-1)|V_{3-i}|-|W|-|S|\nonumber\\
&\geq j\Big(\frac{1}{2}-8\varepsilon^{\frac12}\Big)n
               -(j-1)\Big(\frac12 +2\varepsilon^{\frac12} \Big)n-\frac{3}{2}\varepsilon^{\frac12} n\nonumber\\
&\geq  \Big(\frac12-\big(10k-\frac12\big)\varepsilon^{\frac12}\Big)n\geq \varepsilon^{\frac12}n\nonumber,
\end{align*}
where the last inequality holds by \eqref{EQU-1},
as desired.
\end{proof}

\begin{claim}\label{CLA2.5}
Let $i,j$ be integers with $i\in\{1,2\}$ and $1\leq j\leq k$.
Let $u_0\in V(G)$ with $d_{V_{3-i}}(u_0)\geq 12\varepsilon^{\frac12} n$ and $R\subseteq V(G)\setminus \{u_0\}$ with $|R|\leq k$.\\
(i) For every $u_1\in \overline{V}_{i}\setminus (R\cup \{u_0\})$, $G-R$ contains a $(u_0,u_1)$-path of length $2j$.\\
(ii) For every $u_2\in N_{\overline{V}_{3-i}}(u_0)\setminus R$, $G- R$ contains a $(u_0,u_2)$-path of length $2j-1$.
\end{claim}

\begin{proof}
($i$) Since $\frac12\varepsilon^{\frac12}n\geq |R|+k$,
by Claim \ref{CLA2.4} ($i$),
there exists a $k$-subset $\widetilde{V}_{3-i}\subseteq \overline{V}_{3-i}\setminus R$
such that each vertex in $\widetilde{V}_{3-i}$ is adjacent to $u_0$ and $u_1$.
Clearly, $u_1\in \bigcap_{u\in \widetilde{V}_{3-i}}N_{\overline{V}_{i}}(u)$.
Since $\varepsilon^{\frac12}n\geq |R\cup \{u_0\}|+k$,
by Claim \ref{CLA2.4} ($ii$),
there exists a $k$-subset $\widetilde{V}_{i}$ such that $u_1\in \widetilde{V}_{i}\subseteq \big(\bigcap_{u\in \widetilde{V}_{3-i}}N_{\overline{V}_{i}}(u)\big)\setminus(R\cup \{u_0\})$.
Clearly, $G[\widetilde{V}_{i}, \widetilde{V}_{3-i}]\cong K_{k,k}$.
Take an arbitrary vertex $u_1'\in \widetilde{V}_{3-i}$.
Then, $G[\widetilde{V}_{i}, \widetilde{V}_{3-i}]$ contains a $(u_1',u_1)$-path $P^1$ of length $2j-1$.
Notice that $(\{u_0\}\cup R)\cap (\widetilde{V}_{i}\cup\widetilde{V}_{3-i})=\varnothing$ and  $u_0u_1'\in E(G)$.
Then, the subgraph $G-R$ contains a $(u_0,u_1)$-path of length $2j$.

($ii$) The case $j=1$ is straightforward.
It suffices to consider the case where $j\geq 2$.
Since $\varepsilon^{\frac12}n\geq |R\cup \{u_0\}|+1$,
by Claim \ref{CLA2.4} ($ii$),
 there exists a vertex $u_2'\in  \overline{V}_{i} \setminus (R\cup \{u_0\})$
 adjacent to $u_2$.
Since $\frac12\varepsilon^{\frac12}n\geq |R|+k$,
by Claim \ref{CLA2.4} ($i$),
 there exists a $k$-subset $\widetilde{V}_{3-i}\subseteq \overline{V}_{3-i}\setminus (R\cup \{u_2\})$ such that each vertex in $\widetilde{V}_{3-i}$ is adjacent to $u_0$ and $u_2'$.
Clearly, $u_2'\in \bigcap_{u\in \widetilde{V}_{3-i}}N_{\overline{V}_{i}}(u)$.
Since $\varepsilon^{\frac12}n\geq |R\cup \{u_0\}|+k$,
by Claim \ref{CLA2.4} ($ii$),
there exists a $k$-subset $\widetilde{V}_{i}$ such that $u_2'\in \widetilde{V}_{i}\subseteq \big(\bigcap_{u\in \widetilde{V}_{3-i}}N_{\overline{V}_{i}}(u)\big)\setminus(R\cup \{u_0\})$.
Clearly, $G[\widetilde{V}_{i}, \widetilde{V}_{3-i}]\cong K_{k,k}$.
Take an arbitrary vertex $u_3\in \widetilde{V}_{3-i}$.
Then, $G[\widetilde{V}_{i}, \widetilde{V}_{3-i}]$ contains a $(u_2',u_3)$-path $P^1$ of length $2j-3$.
Notice that $(\{u_0,u_2\}\cup R)\cap (\widetilde{V}_{i}\cup\widetilde{V}_{3-i})=\varnothing$ and $u_0u_3,u_2u_2'\in E(G)$.
Then, we can see that $G-R$ contains a $(u_0,u_2)$-path of length $2j-1$.
\end{proof}

\begin{claim}\label{CLA2.6}
For arbitrary $u_0\in V(G)$ with $d_{G}(u_0)\geq 24\varepsilon^{\frac12} n$,
we have $N_{\overline{V}_j}(u_0)=\varnothing$ for some $j\in \{1,2\}$.
Furthermore, $W\subseteq S$.
\end{claim}

\begin{proof}

Since $d_{G}(u_0)\geq 24\varepsilon^{\frac12} n$, there exists an integer $i\in \{1,2\}$ such that $d_{V_{3-i}}(u_0)\geq 12\varepsilon^{\frac12} n$.
To prove $N_{\overline{V}_j}(u_0)=\varnothing$ for some $j\in \{1,2\}$,
it suffices to prove that $N_{\overline{V}_{i}}(u_0)=\varnothing$.
Otherwise, there exists a vertex $u_1\in N_{\overline{V}_{i}}(u_0)$.
Suppose first that  $r$ is even.
Setting $R=\varnothing$ in Claim \ref{CLA2.5},
$G$ contains a $(u_0,u_1)$-path $P^1$ of length $q$.
Setting $R=V(P^1)\setminus \{u_0,u_1\}$ in Claim \ref{CLA2.5},
$G-R$ contains a $(u_0,u_1)$-path $P^2$ of length $r$.
The subgraph consisting of $\{u_0u_1\}\cup E(P^1)\cup E(P^2)$ is isomorphic to $\theta(1,q,r)$, a contradiction.
Suppose then that $r$ is odd.
By Claim \ref{CLA2.4} ($i$), there exists a vertex $u_2$ in $\overline{V}_{3-i}$ adjacent to $u_0$ and $u_1$.
Setting $R=\{u_1\}$ in Claim \ref{CLA2.5},
$G-R$ contains a $(u_2,u_0)$-path $P^1$ of length $q-1$.
Setting $R=V(P^1)\setminus \{u_2\}$ in Claim \ref{CLA2.5},
$G-R$ contains a $(u_2,u_1)$-path $P^2$ of length $r$.
The subgraph consisting of
$$\{u_2u_1\}\cup (E(P^1)\cup\{u_0u_1\})\cup E(P^2)$$
is isomorphic to $\theta(1,q,r)$, a contradiction.
Thus, $N_{\overline{V}_{i}}(u_0)=\varnothing$.

Now, we prove $W\subseteq S$.
Suppose to the contrary, then we may assume without loss of generality that $W_1\setminus S\neq \varnothing$ and $u\in W_1\setminus S$.
By the definitions of $S$ and $W$, we have
\begin{align*}
d_G(u)>\Big(\frac{1}{2}-4\varepsilon^{\frac12}\Big)n ~~~ \text{and} ~~~ d_{V_1}(u)\geq 4\varepsilon^{\frac12} n.
\end{align*}
Combining this with Claims \ref{CLA2.2} and \ref{CLA2.3}, we obtain
\begin{align*}
|N_{\overline{V}_1}(u)|\geq |N_{V_1}(u)|-|W|-|S|>0.
\end{align*}
On the other hand, since $V(G)=V_1\cup V_2$ is a partition such that $e(V_1,V_2)$ is maximum,
we have $d_{V_1}(u)\leq \frac{1}{2}d_G(u)$.
Combining this with \eqref{EQU-1} gives
\begin{align*}
 d_{V_2}(u)=d_G(u)-d_{V_1}(u)\geq \frac{1}{2}d_G(u)>
 \Big(\frac{1}{4}-2\varepsilon^{\frac12}\Big)n\geq 24\varepsilon^{\frac12}n>|S\cup W|.
\end{align*}
However, this indicates that $N_{\overline{V}_j}(u_0)\neq \varnothing$ for each $j\in \{1,2\}$, a contradiction.
Thus, $W\subseteq S$.
\end{proof}

\begin{claim}\label{CLA2.7}
For any $u_0\in V(G)$,
let $G^{\star}$ be the graph obtained from $G$ by deleting all edges incident to $u_0$,
and joining all possible edges between $u_0$ and $\overline{V}_2\setminus \{u_0\}$.
Then $G^{\star}$ is $\theta(1,q,r)$-free.
\end{claim}

\begin{proof}
Suppose to the contrary, then $G^{\star}$ contains a subgraph $H$ isomorphic to $\theta(1,q,r)$.
From the construction of $G^{\star}$,
we can see that $u_0\in V(H)$.
Assume that $N_{H}(u_0)=\{u_1,u_2,\dots,u_a\}$.
Then $a\leq 3$, and $u_1,u_2,\dots,u_a\in \overline{V}_2$ by the definition of $G^{\star}$.
By Claim \ref{CLA2.4}, we can select a vertex $u\in \overline{V}_1\setminus V(H)$ adjacent to $u_1,u_2,\dots,u_a$.
This implies that $G[(V(H)\setminus \{u_0\})\cup\{u\}]$ already contains a copy of $\theta(1,q,r)$,
a contradiction.
The result follows.
\end{proof}

Now, we are ready to complete the proof of Theorem \ref{THM1.2}.

\begin{proof}
It is not hard to check that both $K_{\lceil\frac{n-2}{2}\rceil,\lfloor\frac{n-2}{2}\rfloor}$ and $K_3$ are $\theta(1,q,r)$-free.
Moreover, $\theta(1,q,r)$ is 2-connected.
This infers that $K_{\lceil\frac{n-2}{2}\rceil,\lfloor\frac{n-2}{2}\rfloor}\circ K_3$ is $\theta(1,q,r)$-free.
Choose an arbitrary graph $G\in{\rm EX}_{3}(n,\theta(1,q,r))$.
Then, $e(G)\geq e(K_{\lceil\frac{n-2}{2}\rceil,\lfloor\frac{n-2}{2}\rfloor}\circ K_3)=\lfloor\frac{(n-2)^2}{4}\rfloor+3\geq (\frac{1}{2}-\delta)\frac{n^2}{2}$ for some constant $\delta>0$.
Then by Lemma \ref{LEM2.3}, $G$ can be obtained from $T_{n,2}$ by adding and deleting at most $\varepsilon n^2$ edges.
Combining these with Definition \ref{DEF2.1},
we can see that $G$ is $\theta(1,q,r)$-good,
and Claims \ref{CLA2.1}-\ref{CLA2.7} hold for $G$.

A shortest odd cycle of $G$ is denoted by $C=w_1w_2\dots w_{g}w_1$.
Suppose that $g\geq 7$.
By Lemma \ref{LEM2.2}, we get $e(G)\leq \lfloor\frac{(n-3)^2}{4}\rfloor+3$,
which contradicts $e(G)\geq \lfloor\frac{(n-2)^2}{4}\rfloor+3$.
Thus, $g\in \{3,5\}$.

\begin{claim}\label{CLA2.8}
$S\subseteq V(C)$.
\end{claim}

\begin{proof}
Suppose to the contrary that there exists a vertex $u_0\in S\setminus V(C)$.
Then, $d_G(u_0)\leq \big(\frac{1}{2}-4\varepsilon^{\frac12}\big)n$.
Setting $u=u_0$ in Claim \ref{CLA2.7},
we can see that $G^{\star}$ is $\theta(1,q,r)$-free.
Moreover, it is clear that $C\subseteq G^{\star}$, which implies that $G^{\star}$ is non-bipartite.

By Claim \ref{CLA2.1}, $|V_2|\geq \big(\frac{1}{2}-2\varepsilon^{\frac12}\big)n$.
Combining this with Claims \ref{CLA2.2} and \ref{CLA2.3}, we obtain
   $$|\overline{V}_2\setminus \{u_0\} |\geq |V_2|-|W|-|S|-1\geq \Big(\frac{1}{2}-\frac{7}{2}\varepsilon^{\frac12}\Big)n-1>d_G(u_0).$$
Consequently, $e(G^{\star})-e(G)=|\overline{V}_2\setminus \{u_0\}|-d_G(u_0)>0$, contradicting the choice of $G$.
\end{proof}

By Claims \ref{CLA2.6} and \ref{CLA2.8}, we have $W\subseteq S \subseteq V(C)$.
Set $G'=G-V(C)$, and $V_i'=V_i\setminus V(C)$ for each $i\in\{1,2\}$.
It follows that $\overline{V}_i\setminus V(C)\subseteq V_i'\subseteq \overline{V}_i$.

\begin{claim}\label{CLA2.9}
For every $u_0\in V(G)$ with $d_{G}(u_0)\geq 24\varepsilon^{\frac12}n$,
we have $N_{G'}(u_0)\subseteq V_i'$ for some $i\in \{1,2\}$.
\end{claim}

\begin{proof}
Since $d_{G}(u_0)\geq 24\varepsilon^{\frac12}n$,
by Claim \ref{CLA2.6},
we have $N_{\overline{V}_{3-i}}(u_0)=\varnothing$ for some $i\in \{1,2\}$.
Since $V_{3-i}'\subseteq \overline{V}_{3-i}$,
it follows that $N_{V_{3-i}'}(u_0)=\varnothing$.
Furthermore, since $V(G')=V_1'\cup V_2'$,
we have $N_{G'}(w_1)\subseteq V_{i}'$, as desired.
\end{proof}

For any $i\in \{1,2\}$ and any $u\in V_i'$,
from \eqref{EQU-3} we know that $d_{V_{3-i}}(u)>\big(\frac{1}{2}-8\varepsilon^{\frac12}\big)n>24\varepsilon^{\frac12}n.$
By Claim \ref{CLA2.9}, we obtain $N_{G'}(u_0)\subseteq V_{3-i}'$,
which implies that $d_{V_i'}(u)=0$.
Thus, $e(V_i')=0$ for every $i\in \{1,2\}$.
Now, we can divide the proof into the following three cases with respect to the values of $q$ and $r$.

\noindent{\textbf{Case 1.} $r$ is even.}

We first prove that $K_{\lceil\frac{n-1}{2}\rceil,\lfloor\frac{n-1}{2}\rfloor}\bullet K_3$ is $\theta(1,q,r)$-free.
Otherwise,  $K_{\lceil\frac{n-1}{2}\rceil,\lfloor\frac{n-1}{2}\rfloor}\bullet K_3$
contains a copy of $\theta(1,q,r)$, say $H$.
Let $u_1$ and $u_2$ be the vertices of degree three in $H$.
Then $u_1u_2\in E(H)$ and $u_1u_2\in E(K_{\lceil\frac{n-1}{2}\rceil,\lfloor\frac{n-1}{2}\rfloor}\bullet K_3)$.
Let $u_3$ be the vertex of degree two in $K_{\lceil\frac{n-1}{2}\rceil,\lfloor\frac{n-1}{2}\rfloor}\bullet K_3$.
Since $d_G(u_i)\geq d_H(u_i)\geq 3$ for each $i\in \{1,2\}$, we have $u_3\notin \{u_1,u_2\}$.
Since $u_1u_2\in E(H)$, it holds that $u_1u_2\in E(G-\{u_3\})$.
We can further observe that every $(u_1,u_2)$-path of even length in $K_{\lceil\frac{n-1}{2}\rceil,\lfloor\frac{n-1}{2}\rfloor}\bullet K_3$
includes the vertex $u_3$ as an internal vertex.
This contradicts the existence of two internally disjoint $(u_1,u_2)$-paths of lengths $q$ and $r$, respectively.
From the choice of $G$ we know that $e(G)\geq e(K_{\lceil\frac{n-1}{2}\rceil,\lfloor\frac{n-1}{2}\rfloor}\bullet K_3)= \lfloor\frac{(n-1)^2}{4}\rfloor+2$.

Recall that $g\in \{3,5\}$.
If $g=5$, then $G$ is triangle-free.
By Lemma \ref{LEM2.2}, $e(G)\leq \lfloor\frac{(n-1)^2}{4}\rfloor+1$, a contradiction.
Thus, $g=3$.
Recall that $C=w_1w_2w_3w_1$.
For any $v\in V(G)$, let $N_{G'}(v)=N_{G}(v)\cap V(G')$ and $d_{G'}(v)=|N_{G'}(v)|$.
Assume without loss of generality that
$q\leq r$ and
$d_{G'}(w_3)=\min_{w\in V(C)}d_{G'}(w)$.

\begin{claim}\label{CLA2.10}
For each $i\in \{1,2\}$, we have $d_{G'}(w_i)\geq \frac{n}{12}$.
\end{claim}

\begin{proof}
Assume that $G^*=G-\{w_1,w_3\}$.
Clearly, $d_{G'}(w_i)= d_G(w_i)-2$ for each $i\in \{1,2,3\}$.
Then,
\begin{align*}
e(G)= e(G^*)+d_{G}(w_1)+d_{G}(w_3)-1= e(G^*)+d_{G'}(w_1)+d_{G'}(w_3)+3.
\end{align*}
Note that $|V(G^*)|=n-2$.
By Lemma \ref{LEM2.1}, $e(G^*)\leq \lfloor\frac{(n-2)^2}{4}\rfloor$.
Combining these with $e(G)\geq\lfloor\frac{(n-1)^2}{4}\rfloor+2$, we obtain
\begin{align*}
2d_{G'}(w_1)\geq d_{G'}(w_1)+d_{G'}(w_3)= e(G)-e(G^*)-3\geq \frac{n}{6},
\end{align*}
which yields that $d_{G'}(w_1)\geq \frac{n}{12}$.
Similarly, $d_{G'}(w_2)\geq \frac{n}{12}$.
\end{proof}

By Claim \ref{CLA2.10}, $d_{G}(w_i)\geq d_{G'}(w_i)\geq 24\varepsilon^{\frac12}n$ for each $i\in \{1,2\}$.
By Claim \ref{CLA2.9},
we may assume without loss of generality that $N_{G'}(w_1)\subseteq V_{1}'$ and $N_{G'}(w_2)\subseteq V_{k_0}'$.

\begin{claim}\label{CLA2.11}
For any $\{i,j\}\subseteq \{1,2,3\}$, we have $N_{G'}(w_i)\cap N_{G'}(w_j)=\varnothing$.
\end{claim}

\begin{proof}
By way of contradiction.
There exists a vertex $u_1\in N_{G'}(w_{i_0})\cap N_{G'}(w_{j_0})$ for some $\{i_0,j_0\}\subseteq \{1,2,3\}$.
If $q=r=2$,
then the subgraph induced by $V(C)\cup \{u_1\}$ contains a copy of $\theta(1,2,2)$, a contradiction.
Then it remains the case  $r\geq 4$.
Take $u_2,u_3\in N_{V_{1}'}(w_1)\setminus \{u_1\}$.

We first consider the case that $i_0=1$ and $j_0=2$.
Clearly, $u_1\in N_{V_{1}'}(w_1)$ as $N_{G'}(w_1)\subseteq V_{1}'$,
which implies that $u_1\in N_{V_{1}'}(w_2)$.
This, together with $N_{G'}(w_2)\subseteq V_{k_0}'$,
gives that $N_{G'}(w_2)\subseteq V_{1}'$.
Choose a vertex $u_4\in N_{V_{1}'}(w_2)\setminus \{u_1,u_2,u_3\}$.
Setting $R=V(C)\cup \{u_1,u_2\}$ in Claim \ref{CLA2.5}, $G-R$ contains a $(u_3,u_4)$-path $P^1$ of length $r-2$.
Setting $R=\{w_2,w_3\}\cup V(P^1)$ in Claim \ref{CLA2.5}, $G-R$ contains a $(w_1,u_1)$-path $P^2$ of length $q-1$.
Furthermore,
the subgraph consisting of
$$\{w_1w_2\}\cup ( E(P^2)\cup\{u_1w_2\}) \cup (\{w_1u_3\}\cup E(P^1)\cup\{u_4w_2\})$$
is isomorphic to $\theta(1,q,r)$, a contradiction.
Hence, $N_{G'}(w_2)\cap N_{G'}(w_1)=\varnothing$.

We then consider the case that $i_0=1$ and $j_0=3$.
Clearly, $u_1\in N_{V_{1}'}(w_1)$  as $N_{G'}(w_1)\subseteq V_{1}'$.
Suppose first that $k_0=2$.
By Claim \ref{CLA2.4} ($i$),
there exists a vertex $v\in \overline{V}_2\setminus V(C)$ adjacent to $w_2$ and $u_2$.
Thus, $v\in \overline{V}_2\setminus V(C)\subseteq V_2'$.
Setting $R=V(C)\cup \{u_1,u_3\}$ in Claim \ref{CLA2.5}, $G-R$ contains a $(v,u_2)$-path $P^1$ of length $r-3$.
Setting $R=\{w_2,w_3\}\cup V(P^1)$ in Claim \ref{CLA2.5}, $G-R$ contains a $(u_1,w_1)$-path $P^2$ of length $q-1$.
Furthermore,
the subgraph consisting of
$$\{w_3w_1\}\cup (\{w_3u_1\}\cup E(P^2)) \cup (\{w_3w_2,w_2v\}\cup E(P^1)\cup\{u_2w_1\})$$
is isomorphic to $\theta(1,q,r)$, a contradiction.

Suppose then that $k_0=1$.
Clearly, there exists a vertex $u_4\in N_{V_{1}'}(w_2)\setminus \{u_1,u_2,u_3\}$.
If $q=2$, then set $R=V(C)$ in Claim \ref{CLA2.5} and we can find that $G-R$ contains a $(u_4,u_1)$-path $P^1$ of length $r-2$.
The subgraph consisting of
$$\{w_1u_1\}\cup \{w_1w_3,w_3u_1\} \cup (\{w_1w_2,w_2u_4\}\cup E(P^1))$$
is isomorphic to $\theta(1,q,r)$, a contradiction.
If $q\geq 4$,
then set $R=V(C)\cup \{u_4\}$ in Claim \ref{CLA2.5} and we can find that $G-R$ contains a $(u_1,u_2)$-path $P^1$ of length $r-2$.
Setting $R=V(C)\cup (V(P^1)\setminus \{u_2\})$ in Claim \ref{CLA2.5},
  $G-R$ contains an $(u_4,u_2)$-path $P^2$ of length $q-2$.
Thus,
the subgraph consisting of
$$\{w_1u_2\} \cup (\{w_1w_2,w_2u_4\}\cup E(P^2))\cup (\{w_1w_3,w_3u_1\}\cup E(P^1))$$
is isomorphic to $\theta(1,q,r)$, a contradiction.
Hence, $N_{G'}(w_3)\cap N_{G'}(w_1)=\varnothing$.

The proof of $N_{G'}(w_3)\cap N_{G'}(w_2)=\varnothing$ is similar to the proof of $N_{G'}(w_3)\cap N_{G'}(w_1)=\varnothing$
 and hence omitted here.
This completes the proof of Claim \ref{CLA2.11}.
\end{proof}

By Claim \ref{CLA2.11}, we have
$$e(V(C),V(G'))=d_{G'}(w_1)+d_{G'}(w_2)+d_{G'}(w_3)\leq |V_{1}'|+|V_{2}'|= n-3.$$
Note that $|V(G')|=n-3$.
By Lemma \ref{LEM2.1}, $e(G')\leq \big\lfloor\frac{(n-3)^2}{4}\big\rfloor$.
Consequently,
\begin{align*}
\Big\lfloor\frac{(n-1)^2}{4}\Big\rfloor+2
&\leq e(G)=e(C)+e(V(C),V(G'))+e(G')\nonumber\\
&\leq 3+(n-3)+\Big\lfloor\frac{(n-3)^2}{4}\Big\rfloor=\Big\lfloor\frac{(n-1)^2}{4}\Big\rfloor+2.
\end{align*}
Thus, $e(G)= \lfloor\frac{(n-1)^2}{4}\rfloor+2$, $G'\cong K_{\lceil\frac{n-3}{2}\rceil,\lfloor\frac{n-3}{2}\rfloor}$, and $G'$ admits a partition $V(G')=\cup_{i=1}^{3}N_{G'}(w_i)$.

By Claim \ref{CLA2.9}, we get  $N_{G'}(w_2)\subseteq V_{1}'$ or $N_{G'}(w_2)\subseteq V_{2}'$.
If $N_{G'}(w_2)\subseteq V_{1}'$,
then by Claim \ref{CLA2.11} and $N_{G'}(w_1)\subseteq V_{1}'$,
we have  $d_{G'}(w_1)+d_{G'}(w_2)\leq |V_1'|$.
This, together with $d_{G'}(w_3)=\min_{w\in V(C)}d_{G'}(w)$ implies that
$$d_{G'}(w_3)\leq  \frac{1}{2}\big(d_{G'}(w_1)+d_{G'}(w_2)\big)
\leq \frac{1}{2}|V_1'|\leq \Big(\frac14 +\varepsilon^{\frac12} \Big)n.$$
On the other hand, since $G'$ admits a partition $V(G')=\cup_{i=1}^{3}N_{G'}(w_i)$,
we have $V_{2}'\subseteq N_{G'}(w_3)$.
By Claim \ref{CLA2.1} and \eqref{EQU-1},
 $$d_{G'}(w_3)\geq |V_2'|\geq \Big(\frac12 -2\varepsilon^{\frac12} \Big)n>\Big(\frac14 +\varepsilon^{\frac12} \Big)n,$$
a contradiction.
Thus, $N_{G'}(w_2)\subseteq V_2'$.
We may assume without loss of generality that $|V_1'|=\lceil\frac{n-3}{2}\rceil$ and $|V_2'|=\lfloor\frac{n-3}{2}\rfloor$.
Since $G'$ admits a partition $V(G')=\cup_{i=1}^{3}N_{G'}(w_i)$,
we get $N_{V_1'}(w_3)=V_1'\setminus N_{V_1'}(w_1)$ and $N_{V_2'}(w_3)=V_2'\setminus N_{V_2'}(w_2)$.

\noindent{\textbf{Subcase 1.1.}} $r=2$.

Then $q=2$ as $q\leq r$.
We first prove that $|N_{V_1'}(w_3)|\cdot|N_{V_2'}(w_3)|\leq 1$.
Otherwise, $|N_{V_1'}(w_3)|\cdot|N_{V_2'}(w_3)|\geq 2$,
then the subgraph $G[N_{G'}(w_3)]$ contains a copy of $P_3$ as $G'\cong K_{\lceil\frac{n-3}{2}\rceil,\lfloor\frac{n-3}{2}\rfloor}$.
Furthermore, $G[\{w_3\}\cup N_{G'}(w_3)]$ contains a copy of $\theta(1,2,2)$, a contradiction.

Conversely, we shall prove that if $|N_{V_1'}(w_3)|\cdot|N_{V_2'}(w_3)|\leq 1$,
then $G$  is $\theta(1,2,2)$-free.
Otherwise, $G$ contains a copy of $\theta(1,2,2)$, say $H$.
Since $G-\{w_3\}$ is bipartite, we get $w_3\in V(H)$.
If $w_3$ is of degree two in $H$, then $G-\{w_3\}$ contains a copy of $H-\{w_3\}\cong C_3$,
which contradicts that $G-\{w_3\}$ is bipartite.
Thus, $w_3$ is of degree three in $H$.
This means that $G[N_G(w_3)]$ contains a copy of $H[N_H(w_3)]\cong P_3$, say $H'$.
Since $N_{G'}(w_i)\cap N_{G'}(w_3)=\varnothing$ for any $i\in \{1,2\}$,
it holds that $G[N_G(w_3)]\cong K_2\cup K_{|N_{V_1'}(w_3)|,|N_{V_2'}(w_3)|}$.
Moreover, since $H'$ is a subgraph of $G[N_G(w_3)]$,
it follows that $|N_{V_1'}(w_3)|\cdot|N_{V_2'}(w_3)|\geq 2$,
a contradiction.

Therefore, by the definition of $\mathcal{G}(n)$, we can see that ${\rm EX}_{3}(n,\theta(1,2,2))=\mathcal{G}(n)$.

\noindent{\textbf{Subcase 1.2.}} $r\geq 4$  is even.

We first prove that $N_{V_1'}(w_3)=\varnothing$.
Suppose to the contrary, then there exists a vertex $u_1\in N_{V_1'}(w_3)$.
Choose vertices $u_2,u_3\in N_{V_1'}(w_1)$.
By Claim \ref{CLA2.4} ($i$), there exists a vertex $v\in \overline{V}_2\setminus V(C)$ adjacent to $w_2$ and $u_3$.
Thus, $v\in \overline{V}_2\setminus V(C)\subseteq V_2'$.
Setting $R=V(C)\cup \{u_3,v\}$ in Claim \ref{CLA2.5}, $G-R$ contains a $(u_1,u_2)$-path $P^1$ of length $r-2$.
If $q=2$, then the subgraph consisting of
$$\{w_3w_1\}\cup (\{w_3w_2,w_2w_1\}) \cup (\{w_3u_1\}\cup E(P^1)\cup\{u_2w_1\})$$
is isomorphic to $\theta(1,q,r)$, a contradiction.
If $q\geq 4$, then set $R=V(C)\cup V(P^1)$ in Claim \ref{CLA2.5} and we can find that $G-R$ contains a $(v,u_3)$-path $P^2$ of length $q-3$.
Furthermore,
the subgraph consisting of
$$\{w_3w_1\}\cup (\{w_3w_2,w_2v\}\cup E(P^2)\cup\{u_3w_1\}) \cup (\{w_3u_1\}\cup E(P^1)\cup\{u_2w_1\})$$
is isomorphic to $\theta(1,q,r)$, a contradiction.
Thus, $N_{V_1'}(w_3)=\varnothing$.
Similarly, $N_{V_2'}(w_3)=\varnothing$.

Since $N_{V_1'}(w_3)=N_{V_2'}(w_3)=\varnothing$,
we have $N_{G'}(w_1)=V_1'$ and $N_{G'}(w_2)=V_2'$.
Therefore, $G\cong K_{\lceil\frac{n-1}{2}\rceil,\lfloor\frac{n-1}{2}\rfloor}\bullet K_3$.

\noindent{\textbf{Case 2.}} $r$ is odd and $q=2$.

Clearly, $SK_{\lceil\frac{n-1}{2}\rceil,\lfloor\frac{n-1}{2}\rfloor}$ is $C_{3}$-free, and hence $\theta(1,q,r)$-free.
Thus, $e(G)\geq e(SK_{\lceil\frac{n-1}{2}\rceil,\lfloor\frac{n-1}{2}\rfloor})= \lfloor\frac{(n-1)^2}{4}\rfloor+1$.
Recall that $g\in \{3,5\}$.

\noindent{\textbf{Subcase 2.1.}} $g=3$.

Then, $C=w_1w_2w_3w_1$.
Assume without loss of generality that $d_{G'}(w_3)=\min_{w\in V(C)}d_{G'}(w)$.

\begin{claim}\label{CLA2.12}
$d_{G'}(w_i)\geq \frac{n}{12}$  for $i\in \{1,2\}$.
\end{claim}

\begin{proof}
Assume that $G^*=G-\{w_1,w_3\}$.
Clearly, $d_{G'}(w_i)= d_G(w_i)-2$ for each $i\in \{1,2,3\}$.
Then,
\begin{align*}
e(G)= e(G^*)+d_{G}(w_1)+d_{G}(w_3)-1=  e(G^*)+d_{G'}(w_1)+d_{G'}(w_3)+3.
\end{align*}
Note that $|V(G^*)|=n-2$.
By Lemma \ref{LEM2.1}, $e(G^*)\leq \lfloor\frac{(n-2)^2}{4}\rfloor$.
Combining this with $e(G)\geq\lfloor\frac{(n-1)^2}{4}\rfloor+1$, we obtain
\begin{align*}
2d_{G'}(w_1)\geq d_{G'}(w_1)+d_{G'}(w_3)=e(G)-e(G^*)-3\geq \frac{n}{3},
\end{align*}
which yields that $d_{G'}(w_1)\geq \frac{n}{12}$.
Similarly, $d_{G'}(w_2)\geq \frac{n}{12}$.
\end{proof}

By Claim \ref{CLA2.9}, we may assume without loss of generality that  $N_{G'}(w_1)\subseteq V_{1}'$.

\begin{claim}\label{CLA2.13}
(i) For any integer $i\in \{1,2,3\}$, $N_{G'}(w_i)\subseteq V_1'$.\\
(ii) For any $\{i,j\}\subset \{1,2,3\}$, $N_{V_1'}(w_i)\cap N_{V_1'}(w_j)=\varnothing$.
\end{claim}

\begin{proof}
($i$) By way of contradiction.
Then, there exists a vertex $u_2\in N_{V_2'}(w_{i_0})$ for some $i_0\in \{2,3\}$ as $N_{G'}(w_1)\subseteq V_1'$.
By Claim \ref{CLA2.4} ($i$),
there exists a vertex $u_1\in \overline{V}_1\setminus V(C)$ adjacent to $w_1$ and $u_2$. Thus, $u_1\in \overline{V}_1\setminus V(C)\subseteq V_1'$.
Setting $R=V(C)$ in Claim \ref{CLA2.5},
$G-R$ contains a $(u_1,u_2)$-path $P^1$ of length $r-2$.
The subgraph consisting of $$\{w_1w_{i_0}\}\cup \{w_1w_{5-{i_0}},w_{5-{i_0}}w_{i_0}\}\cup
(\{w_1u_1\}\cup E(P^1)\cup\{u_2w_{i_0}\})$$ is isomorphic to $\theta(1,2,r)$, a contradiction.
Hence,  $N_{G'}(w_i)\subseteq V_1'$ for any $i\in \{1,2,3\}$.

($ii$) We first prove that $N_{V_1'}(w_1)\cap N_{V_1'}(w_i)=\varnothing$  for each $i\in \{2,3\}$.
Otherwise, there exists a vertex $u_1\in N_{V_1'}(w_1)\cap N_{V_1'}(w_{i_0})$ for some $i_0\in \{2,3\}$.
By Claim \ref{CLA2.12},
we choose a vertex $u_2\in N_{V_{1}'}(w_1)\setminus \{u_1\}$.
Set $R=V(C)$ in Claim \ref{CLA2.5} and we can find that $G-R$ contains a $(u_2,u_1)$-path $P^1$ of length $r-1$.
Then the subgraph consisting of
$$\{w_1u_1\}\cup \{w_1w_{i_0},w_{i_0}u_1\} \cup (\{w_1u_2\}\cup E(P^1))$$
is isomorphic to $\theta(1,2,r)$, a contradiction.
Hence, $N_{G'}(w_1)\cap N_{G'}(w_i)=\varnothing$ for each $i\in \{1,2\}$.
Similarly, we can also get $N_{G'}(w_2)\cap N_{G'}(w_3)=\varnothing$, as desired.
\end{proof}

By Claims \ref{CLA2.13} and \ref{CLA2.1}, we have
\begin{align}\label{EQU-4}
 e(V(C),V(G'))=\sum\limits_{i=1}^{3}d_{V_1'}(w_i)\leq |V_1'|\leq \Big(\frac12 +2\varepsilon^{\frac12} \Big)n.
\end{align}
Note that $|V(G')|=n-3$.
By Lemma \ref{LEM2.1}, $e(G')\leq \lfloor\frac{(n-3)^2}{4}\rfloor$.
Then,
\begin{align*}
 e(V(C),V(G'))=e(G)-e(C)-e(G')\geq \Big\lfloor\frac{(n-1)^2}{4}\Big\rfloor+1-3-\Big\lfloor\frac{(n-3)^2}{4}\Big\rfloor=n-4,
\end{align*}
which contradicts \eqref{EQU-4}.
Hence, $g\neq 3.$

\noindent{\textbf{Subcase 2.2.}} $g=5$.

Then, $G$ is $C_3$-free.
From Lemma \ref{LEM2.2} we know that $e(G)\leq\lfloor\frac{(n-1)^2}{4}\rfloor+1$.
This, together with $e(G)\geq\lfloor\frac{(n-1)^2}{4}\rfloor+1$,
implies that $e(G)=\lfloor\frac{(n-1)^2}{4}\rfloor+1$.
Again by Lemma  \ref{LEM2.2}, we can observe that $G\in\mathcal{H}(n)$ as $G$ is $C_3$-free.
That is, ${\rm EX}_{3}(n,\theta(1,q,r))\subseteq \mathcal{H}(n)$.
On the other hand, any graph in $\mathcal{H}(n)$ is $C_3$-free,
and hence $\theta(1,q,r)$-free.
Therefore, ${\rm EX}_{3}(n,\theta(1,q,r))=\mathcal{H}(n)$, as desired.

\noindent{\textbf{Case 3.}} $r$ is odd and $q\geq 4$.

Recall that $K_{\lceil\frac{n-2}{2}\rceil,\lfloor\frac{n-2}{2}\rfloor}\circ K_3$ is $\theta(1,q,r)$-free,
and $e(G)\geq\lfloor\frac{(n-2)^2}{4}\rfloor+3$.

\noindent{\textbf{Subcase 3.1.}} $g=5$.

We first give the following two claims.

\begin{claim}\label{CLA2.14}
For any edge $v_1v_2\in E(C)$, we have $d_{G'}(v_i)<\frac{n}{18}$ for some $i\in \{1,2\}$.
\end{claim}

\begin{proof}
Suppose to the contrary that $d_{G'}(v_1)\geq \frac{n}{18}$ and $d_{G'}(v_2)\geq \frac{n}{18}$.
By Claim \ref{CLA2.9}, we may assume without loss of generality that  $N_{G'}(v_1)\subseteq V_{1}'$
and $N_{G'}(v_2)\subseteq V_{i_0}'$.

Suppose first that $i_0=1$, that is, $N_{G'}(v_2)\subseteq V_{1}'$.
Choose $u_1,u_2\in N_{V_{1}'}(v_1)$ and $u_3\in N_{V_{1}'}(v_2)\setminus \{u_1,u_2\}$.
Setting $R=\{v_1,v_2,u_3\}$ in Claim \ref{CLA2.5},
 $G-R$ contains a $(u_2,u_1)$-path $P^1$ of length $r-1$.
Set $R=\{v_1,v_2\}\cup (V(P^1)\setminus \{u_1\})$ in Claim \ref{CLA2.5} and we can find that $G-R$ contains an $(u_3,u_1)$-path $P^2$ of length $q-2$.
Then the subgraph consisting of
$$\{v_1u_1\}\cup (\{v_1v_2,v_2u_3\}\cup E(P^2))\cup (\{v_1u_2\}\cup E(P^1)) $$
is isomorphic to $\theta(1,q,r)$, a contradiction.

Suppose next that $i_0=2$, that is, $N_{G'}(v_2)\subseteq V_{2}'$.
We may assume $C=v_1v_2v_3v_4v_5v_1$.
Take $u_1,u_2\in N_{V_{1}'}(v_1)$.
By Claim \ref{CLA2.4} ($i$),
there exists a vertex $u_3\in \overline{V}_2\setminus V(C)$ adjacent to $v_2$ and $u_1$.
Thus, $u_3\in V_2'$.
If $q=4$, then set $R=V(C)$ in Claim \ref{CLA2.5} and we can find that $G-R$ contains a $(u_1,u_3)$-path $P^1$ of length $r-2$.
Then the subgraph consisting of
$$\{v_1v_2\}\cup \{v_1v_5,v_5v_4,v_4v_3,v_3v_2\}\cup (\{v_1u_1\}\cup E(P^1)\cup\{u_3v_2\}) $$
is isomorphic to $\theta(1,q,r)$, a contradiction.
If $q\geq 6$, then set $R=V(C)\cup \{u_2\}$ in Claim \ref{CLA2.5} and we can find that $G-R$ contains a $(u_3,u_1)$-path $P^1$ of length $q-5$.
Set $R=V(C)\cup (V(P^1)\setminus \{u_1\})$ in Claim \ref{CLA2.5} and we can find that $G-R$ contains an $(u_2,u_1)$-path $P^2$ of length $r-1$.
Then the subgraph consisting of
$$\{v_1u_1\}\cup (\{v_1v_5,v_5v_4,v_4v_3,v_3v_2,v_2u_3\}\cup E(P^1))\cup (\{v_1u_2\}\cup E(P^2)) $$
is isomorphic to $\theta(1,q,r)$, a contradiction.
The claim holds.
\end{proof}

\begin{claim}\label{CLA2.15}
For arbitrary  $v_1,v_2,v_3\in V(C)$ with $d_{G'}(v_1)\geq d_{G'}(v_2)\geq d_{G'}(v_3)$,
$d_{G'}(v_1)\geq \frac{n}{18}$.
\end{claim}

\begin{proof}
Assume that $G^*=G-\{v_1,v_2,v_3\}$.
Clearly, $d_{G'}(v_i)=d_G(v_i)-d_C(v_i)=d_G(v_i)-2$ for each $i\in \{1,2,3\}$.
Then,
\begin{align*}
e(G)\leq  e(G^*)+\sum_{i=1}^3d_{G}(v_i)=  e(G^*)+\sum_{i=1}^3d_{G'}(v_i)+6.
\end{align*}
Note that $|V(G^*)|=n-3$.
By Lemma \ref{LEM2.1}, $e(G^*)\leq \lfloor\frac{(n-3)^2}{4}\rfloor$.
Combining this with $e(G)\geq\lfloor\frac{(n-2)^2}{4}\rfloor+3$, we obtain
\begin{align*}
3d_{G'}(v_1)\geq \sum_{i=1}^3d_{G'}(u_i)\geq e(G)-e(G^*)-6\geq \frac{n}{3},
\end{align*}
which yields that $d_{G'}(v_1)\geq \frac{n}{18}$.
\end{proof}

By Claim \ref{CLA2.15}, it is not hard to observe that there are at least three vertices $v\in V(C)$ satisfying $d_{G'}(v)\geq \frac{n}{18}$.
Among these three vertices, there must be two adjacent vertices in $C$,
contradicting Claim \ref{CLA2.14}.
Therefore, $g\neq 5$.

\noindent{\textbf{Subcase 3.2.}} $g=3$.

Then $C=w_1w_2w_3w_1$.
Assume without loss of generality that $d_{G'}(w_3)=\max_{w\in V(C)}d_{G'}(w)$.
By Lemma \ref{LEM2.1}, $e(G')\leq \lfloor\frac{(n-3)^2}{4}\rfloor$.
Then,
\begin{align*}
 3d_{G'}(w_3)
 &\geq e(V(C),V(G'))=e(G)-e(C)-e(G')\\
 &\geq \Big\lfloor\frac{(n-2)^2}{4}\Big\rfloor+3-3-\Big\lfloor\frac{(n-3)^2}{4}\Big\rfloor
 \geq \frac{n}{3},
\end{align*}
which yields that $d_{G'}(w_3)\geq \frac{n}{9}$.
By Claim \ref{CLA2.9}, we may assume without loss of generality that $N_{G'}(w_3)\subseteq V_{1}'$.

\begin{claim}\label{CLA2.16}
For every $i\in \{1,2\}$, we have $N_{G'}(w_i)=\varnothing$.
\end{claim}

\begin{proof}
By way of contradiction. We may assume that there exists an integer $i_0\in \{1,2\}$ such that $N_{G'}(w_{i_0})\neq\varnothing$,
more precisely, there exists a vertex $u_2\in N_{G'}(w_{i_0})$.

Suppose first that $u_2\in V_2'$.
By Claim \ref{CLA2.4} ($i$), there exist two vertices $u_1,u_3$ in $\overline{V}_1\setminus V(C)$ adjacent to $w_3$ and $u_2$. Then, $u_1,u_3\in V_1'$.
Setting $R=V(C)\cup \{u_3\}$ in Claim \ref{CLA2.5},
$G-R$ contains a $(u_1,u_2)$-path $P^1$ of length $r$.
If $q=4$, then the subgraph consisting of
$$\{u_1u_2\}\cup \{u_1w_3,w_3w_{3-i_0},w_{3-i_0}w_{i_0},w_{i_0}u_2\}\cup E(P^1)$$
is isomorphic to $\theta(1,q,r)$, a contradiction.
If $q\geq 6$, then set $R=V(C)\cup (V(P^1)\setminus \{u_1\})$ in Claim \ref{CLA2.5} and we can find that
$G-R$ contains a $(u_1,u_3)$-path $P^2$ of length $q-4$.
Consequently, the subgraph consisting of
$$\{u_1u_2\}\cup (E(P^2)\cup\{u_3w_3,w_3w_{3-i_0},w_{3-i_0}w_{i_0},w_{i_0}u_2\})\cup E(P^1)$$
is isomorphic to $\theta(1,q,r)$, a contradiction.

Thus, $u_2\in V_1'$.
Choose  vertices $u_1,u_3\in N_{V_1'}(w_3)\setminus \{u_2\}$.
Setting $R=V(C)\cup \{u_3\}$ in Claim \ref{CLA2.5},
$G-R$ contains a $(u_2,u_1)$-path $P^1$ of length $q-2$.
Setting $R=V(C)\cup (V(P^1)\setminus \{u_1\})$ in Claim \ref{CLA2.5},
$G-R$ contains a $(u_3,u_1)$-path $P^2$ of length $r-1$.
Consequently, the subgraph consisting of
$$\{w_3u_1\}\cup (\{w_3w_{i_0},w_{i_0}u_2\}\cup E(P^1))\cup (\{w_3u_3\}\cup E(P^2))$$
is isomorphic to $\theta(1,q,r)$, which also gives a contradiction.
Thus, the claim holds.
\end{proof}

By Claim \ref{CLA2.16}, we get $d_{G'}(w_i)=0$ for $i\in \{1,2\}$.
Note that $N_{G'}(w_3)\subseteq V_1'$.
Then, $G-\{w_1,w_2\}$ is a subgraph of $K_{|V_1'|,|V_2'|+1}$.
Since $|V_1'|+|V_2'|+1=n-2$, we have
$$e(G)=e(G-\{w_1,w_2\})+3\leq |V_1'|(|V_2'|+1)+3\leq \Big\lfloor\frac{(n-2)^2}{4}\Big\rfloor+3.$$
Combining this with $e(G)\geq \lfloor\frac{(n-2)^2}{4}\rfloor+3$ gives that
$G-\{w_1,w_2\}\cong K_{\lfloor\frac{n-2}{2}\rfloor,\lceil\frac{n-2}{2}\rceil}$.
We can further obtain that  $G$ is isomorphic to either $K_{\lceil\frac{n-2}{2}\rceil, \lfloor\frac{n-2}{2}\rfloor}\circ K_3$ or $K_{\lfloor\frac{n-2}{2}\rfloor,\lceil\frac{n-2}{2}\rceil}\circ K_3$, as desired.

This completes the proof of Theorem \ref{THM1.2}.
\end{proof}

\section{Proof of Theorem \ref{THM1.3}}\label{section4}

In this section, we first list the preparatory lemmas.
%The following result is due to Li and Ning \cite{Li-N2023}.

\begin{lem}\label{lem3.3}\emph{(\cite{Li-N2023})}
Let $G$ be a graph. For any $u\in V(G)$, $\rho^2(G)\leq \rho^2(G-\{u\})+2d_G(u).$
\end{lem}

%The following lemma was given in \cite{Wu2005}.

%\begin{lem}\label{lem3.8}\emph{(\cite{Wu2005})}
%Let $G$ be a connected graph and $\mathbf{x}$ be the Perron vector of $G$.
%Assume that $u,v$ are two vertices of $G$ with $x_u\geq x_v$ and $\{v_i~|~1\leq i \leq s\}\subseteq N_G(v)\setminus (N_G(u)\cup \{u\})$.
%Let $G^*=G-\{vv_i~|~1\leq i \leq s\}+\{uv_i~|~1\leq i \leq s\}$.
%Then, $\rho(G^*)>\rho(G)$.
%\end{lem}

\begin{lem}\label{lem3.2}\emph{(\cite{Guo20212})}
Let $a,b\geq 2$ be integers with $a+b=n-2$.
If $n\geq 10$, then $\rho(K_{a,b}\circ K_3)\leq\rho(K_{\lceil\frac{n-2}{2}\rceil,\lfloor\frac{n-2}{2}\rfloor}\circ K_3)$
with equality if and only if $K_{a,b}\circ K_3\cong K_{\lceil\frac{n-2}{2}\rceil,\lfloor\frac{n-2}{2}\rfloor}\circ K_3$.
\end{lem}

%Lin, Ning and Wu \cite{Lin-N2021} proved a generalization on spectral
%Mantel theorem for non-bipartite graphs.

\begin{lem}\label{lem3.1}
(i) \emph{(\cite{Lin-N2021})}
$\mathrm{SPEX}_{3}(n,C_3)=\{ SK_{\lceil\frac{n-1}{2}\rceil,\lfloor\frac{n-1}{2}\rfloor}\}$.\\
(ii)  \emph{(\cite{Guo20212})} $\mathrm{SPEX}_{3}(n,C_5)=\{K_{\lceil\frac{n-2}{2}\rceil,\lfloor\frac{n-2}{2}\rfloor}\circ K_3\}$ for $n\geq 21$.
\end{lem}

%\begin{lem}\label{lem3.1}\emph{(\cite[Proposition 3.3]{Lin-N2021})}
%Let $a,b,n$ be positive integers with $a+b=n-1$.
%Then $\rho(SK_{a,b})\leq \rho(SK_{\lceil\frac{n-1}{2}\rceil,\lfloor\frac{n-1}{2}\rfloor})$,
%with equality if and only if
%$(a,b)=(\lceil\frac{n-1}{2}\rceil,\lfloor\frac{n-1}{2}\rfloor)$.
%\end{lem}

The following is the spectral version of the Stability Lemma due to Nikiforov \cite{Nikiforov-2009}.

\begin{lem}\label{thm0}\emph{(\cite{Nikiforov-2009})}
Let $r\ge 2$, $\frac{1}{\ln n}<c<r^{-8(r+21)(r+1)}$, $0<\varepsilon<2^{-36}r^{-24}$ and $G$ be an $n$-vertex graph.
If $\rho(G)>(1-\frac1r-\varepsilon)n$, then one of the following holds:\\
(i) $G$ contains a $K_{r+1}(\lfloor c\ln n\rfloor, \dots,\lfloor c\ln n\rfloor,\lceil n^{1-\sqrt{c}}\rceil)$;\\
(ii) $G$ differs from $T_{n,r}$ in fewer than $(\varepsilon^{\frac{1}{4}}+c^{\frac{1}{8r+8}})n^2$ edges.
\end{lem}

%By Nikiforov's result \cite[Theorem 2]{Nikiforov-2009-2}   and a more careful calculation on the
%equality case in his proof,
%one can obtain the following spectral version of  the color-critical theorem, which was given by Zhai and Lin \cite[Theorem 1.2]{Zhai2023}.

From Lemma \ref{thm0}, Desai et al. \cite{Desai-2022} derived the following stability result.
Lemma \ref{thm0} and the subsequent lemma provide an effective approach for studying spectral extremal problems.

\begin{lem} \label{lem3.5}\emph{(\cite{Desai-2022})}
Let $F$ be a graph with  $\chi(F)=r+1$.
For every $\varepsilon>0$, there exist $\delta>0$ and $n_0$ such that
if $G$ is an $F$-free graph on $n\ge n_0$ vertices with $\rho(G)\geq (1-\frac1r-\delta)n$,
then $G$ can be obtained from $T_{n,r}$ by adding and deleting at most $\varepsilon n^2$ edges.
\end{lem}

Based on Nikiforov's result \cite[Theorem 2]{Nikiforov-2009-2} and a more detailed analysis of the equality case in his proof, one can derive the following spectral version of the color-critical theorem, as presented by Zhai and Lin \cite[Theorem 1.2]{Zhai2023}.

\begin{lem} \label{lem3.4}\emph{(\cite{Nikiforov-2009-2, Zhai2023})}
Let $r\geq 2$ and $H$ be a color-critical graph with $\chi(H)=r+1$.
Then there exists an $n_0(H)\geq e^{|V(H)|r^{(2r+9)(r+1)}}$ such that
 ${\rm SPEX}(n,H)=\{T_{n,r}\}$ provided $n\geq n_0(H)$.
\end{lem}

Now we are ready to give the proof of Theorem \ref{THM1.3}.

\begin{proof}
Choose an arbitrary graph $G\in {\rm SPEX}_{3}(n,\theta(1,q,r))$.
Clearly, $G$ is connected.
Otherwise, we first choose two components $G_1$ and $G_2$ with $\rho(G_1)=\rho(G),$ and then add an edge between $G_1$ and $G_2$ to obtain a new graph with larger spectral radius, which gives a contradiction.
By the Perron-Frobenius theorem,
there exists a unique unit positive eigenvector
$\mathbf{x}=(x_1,\ldots,x_n)^{\mathsf{T}}$ corresponding to $\rho(G)$,
where $x_u$ is the coordinate of $\mathbf{x}$ corresponding to the vertex $u$ of $G$.
We shall refer to such an eigenvector as the \emph{Perron vector} of $G$.
Set $x_{u^*}=\max\{x_i: i\in V(G)\}$.
Let $\varepsilon$ be defined as in \eqref{EQU-1}.
Recall that $K_{\lceil\frac{n-2}{2}\rceil,\lfloor\frac{n-2}{2}\rfloor}\circ K_3$ is $\theta(1,q,r)$-free.
Hence,
\begin{align}\label{EQU-5}
\rho(G)&\geq \rho(K_{\lceil\frac{n-2}{2}\rceil,\lfloor\frac{n-2}{2}\rfloor}\circ K_3)
> \rho(K_{\lceil\frac{n-2}{2}\rceil,\lfloor\frac{n-2}{2}\rfloor})\geq \frac{\mathbf{1}^{\mathsf{T}}A(K_{\lceil\frac{n-2}{2}\rceil,\lfloor\frac{n-2}{2}\rfloor})\mathbf{1}}
{\mathbf{1}^{\mathsf{T}}\mathbf{1}}\nonumber\\
  &=\frac{2e(K_{\lceil\frac{n-2}{2}\rceil,\lfloor\frac{n-2}{2}\rfloor})}{n-2}>\frac{n-2.1}{2}
    \geq \max\left\{\Big(\frac12-\delta\Big)n,\Big(\frac12-\varepsilon^{\frac12}\Big)n\right\}.
\end{align}
Then by Lemma \ref{lem3.5}, $G$ can be obtained from $T_{n,2}$
by adding and deleting at most $\varepsilon n^2$ edges.

A shortest odd cycle of $G$ is denoted by $C=w_1w_2\dots w_{g}w_1$.
If $g\geq 7$, then $G$ is $C_5$-free.
By Lemma \ref{lem3.1} ($ii$),
 $\rho(G)\leq \rho(K_{\lceil\frac{n-2}{2}\rceil,\lfloor\frac{n-2}{2}\rfloor}\circ K_3)$,
with equality if and only if $G\cong K_{\lceil\frac{n-2}{2}\rceil,\lfloor\frac{n-2}{2}\rfloor}\circ K_3$.
Recall that $\rho(G)\geq \rho(K_{\lceil\frac{n-2}{2}\rceil,\lfloor\frac{n-2}{2}\rfloor}\circ K_3)$.
Thus, $G\cong K_{\lceil\frac{n-2}{2}\rceil,\lfloor\frac{n-2}{2}\rfloor}\circ K_3$.
However,  $K_{\lceil\frac{n-2}{2}\rceil,\lfloor\frac{n-2}{2}\rfloor}\circ K_3$ contains a copy of $C_3$,
contradicting $g\geq 7$.
Thus, $g\in \{3,5\}$.
By Definition \ref{DEF2.1},
$G$ is $\theta(1,q,r)$-good,
and thus Claims \ref{CLA2.1}-\ref{CLA2.7} hold for $G$.

\begin{claim}\label{CLA3.1}
$S\subseteq V(C)$.
\end{claim}

\begin{proof}
Assume without loss of generality that $u^*\in V_1$.
Since $x_{u^*}=\max\{x_i: i\in V(G)\}$,
we have $\rho(G)x_{u^*}=\sum_{v\in N_{G}(u^*)}x_v\leq d_G(u^*)x_{u^*}$.
This yields that $d_G(u^*)\geq \rho(G)\geq (\frac12-\varepsilon^{\frac12})n$.
Thus, $u^*\notin S$.
From Claim \ref{CLA2.6} we know that $W\subseteq S$, and hence
$u^*\in \overline{V}_1$.
Clearly, $d_G(u^*)\geq (\frac12-\varepsilon^{\frac12})n\geq 24\varepsilon^{\frac12} n$.
Again by Claim \ref{CLA2.6}, we get $N_{\overline{V}_1}(u^*)=\varnothing$.
Combining these with Claim \ref{CLA2.2} gives
\begin{align}\label{EQU-6}
\rho(G)x_{u^*}&=\!\!\sum_{v\in N_{S}(u^*)}x_v+\sum_{v\in N_{\overline{V}_2}(u^*)}x_v
\leq \varepsilon^{\frac12} n x_{u^*}+\sum_{v\in \overline{V}_2}x_v.
\end{align}

Suppose to the contrary that there exists a vertex $u_0\in S\setminus V(C)$.
Setting $u=u_0$ in Claim \ref{CLA2.7},
we can see that $G^{\star}$ is $\theta(1,q,r)$-free.
Moreover, it is clear that $C\subseteq G^{\star}$, which implies that $G^{\star}$ is non-bipartite.
In what follows, we shall obtain a contradiction by showing that $\rho(G^{\star})>\rho(G)$.
From the definition of $S$ we know that $d_G(u_0)\leq \big(\frac12-4\varepsilon^{\frac12}\big)n$.
Combining \eqref{EQU-5} and \eqref{EQU-6},
we have
\begin{align}\label{EQU-7}
\sum_{v\in \overline{V}_2\setminus \{u_0\}}x_v\geq (\rho(G)-\varepsilon^{\frac12} n-1)x_{u^*}
>\Big(\frac12-2\varepsilon^{\frac12}\Big)n x_{u^*}>d_G(u_0)x_{u^*}.
\end{align}
Consequently,
\begin{center}
  $\rho(G^{\star})-\rho(G) \geq  \mathbf{x}^{\mathsf{T}}\big(A(G^{\star})-A(G)\big)\mathbf{x}
                  = 2x_{u_0}\Big(\sum\limits_{v\in \overline{V}_2\setminus \{u_0\}}x_v-\sum\limits_{v\in N_G(u_0)}x_v\Big)>0,$
\end{center}
contradicting that $G\in {\rm SPEX}_{3}(n,\theta(1,q,r))$.
Hence, $S\subseteq V(C)$.
\end{proof}

\begin{claim}\label{CLA3.2}
For every $u\in V(G')$, we have $x_u\geq  (1-4\varepsilon^{\frac12})x_{u^*}$.
\end{claim}

\begin{proof}
Assume without loss of generality that $u^*\in V_1$.
By \eqref{EQU-6} and Claim \ref{CLA2.1}, we have
   $$\rho(G)x_{u^*}<\varepsilon^{\frac12} n x_{u^*}+\sum_{v\in \overline{V}_2}x_v
   \leq \varepsilon^{\frac12} nx_{u^*}+|V_2|x_{u^*}
   \leq \Big(\frac12+3\varepsilon^{\frac12}\Big) nx_{u^*}.$$

Now we show $x_u\geq (1-4\varepsilon^{\frac12})x_{u^*}$ for each $u\in V(G)\setminus S$.
Suppose to the contrary that $x_{u_0}<(1-4\varepsilon^{\frac12})x_{u^*}$
for some $u_0\in V(G)$.
Setting $u=u_0$ in Claim \ref{CLA2.7},
we can see that $G^{\star}$ is $\theta(1,q,r)$-free.
Moreover, it is clear that $C\subseteq G^{\star}$, which implies that $G^{\star}$ is non-bipartite.

In view of \eqref{EQU-5},
we have $\rho(G)>\frac{n-2.1}{2}>\frac{n}{3}$.
Combining \eqref{EQU-7},
we obtain
\begin{align*}
\sum_{v\in \overline{V}_2\setminus\{u_0\}}x_v\!\!-\!\!\sum_{v\in N_G(u_0)}x_{v}
&=\sum_{v\in \overline{V}_2\setminus\{u_0\}}x_v\!\!-\!\!\rho(G)x_{u_0}\nonumber\\
&\geq(\rho(G)-\varepsilon^{\frac12} n-1)x_{u^*}
-\rho(G)\Big(1-4\varepsilon^{\frac12}\Big)x_{u^*}   \nonumber\\
&\geq (4\varepsilon^{\frac12}\rho(G)-\varepsilon^{\frac12} n-1)x_{u^*}
\geq \Big(\frac43\varepsilon^{\frac12}n-\varepsilon^{\frac12} n-1\Big)x_{u^*}>0.  \nonumber
\end{align*}
Consequently,
\begin{align*}
\rho(G^*)-\rho(G)
\geq \mathbf{x}^{\mathsf{T}}\big(A(G^*)-A(G)\big)\mathbf{x}
=2x_{u_0}\Big(\sum_{v\in \overline{V}_2\setminus\{u_0\}}x_v
-\sum_{v\in N_G(u_0)}x_{v}\Big)>0,
\end{align*}
contradicting that $G\in {\rm SPEX}_{3}(n,\theta(1,q,r))$.
\end{proof}

By Claims \ref{CLA2.6} and \ref{CLA3.1}, we have $W\subseteq S \subseteq V(C)$.
Set $G'=G-V(C)$, and $V_i'=V_i\setminus V(C)$ for $i\in\{1,2\}$.
It follows that $\overline{V}_i\setminus V(C)\subseteq V_i'\subseteq \overline{V}_i$.
For any $i\in \{1,2\}$ and any $u\in V_i'$,
from \eqref{EQU-3} we know that $d_{V_{3-i}}(u)>(\frac{1}{2}-8\varepsilon^{\frac12})n>24\varepsilon^{\frac12}n.$
Thus, by Claim \ref{CLA2.6}, $d_{V_i'}(u)=0$.
We can further obtain that $e(V_i')=0$ for every $i\in \{1,2\}$.
Using a similar argument as in the proof of Claim \ref{CLA2.9},
we obtain the following claim.

\begin{claim}\label{CLA3.3}
For every $u_0\in V(G)$ with $d_{G}(u_0)\geq 24\varepsilon^{\frac12}n$,
we have $N_{G'}(u_0)\subseteq V_i'$ for some $i\in \{1,2\}$.
\end{claim}

%\begin{proof}
%Clearly, $d_{G}(u_0)\geq 24\varepsilon^{\frac12}n$.
%By Claim \ref{CLA2.6},
%we have $N_{\overline{V}_{3-i}}(u_0)=\varnothing$ for some $i\in \{1,2\}$.
%Since $V_{3-i}'\subseteq \overline{V}_{3-i}$,
%it follows that $N_{V_{3-i}'}(u_0)=\varnothing$.
%Furthermore, since $V(G')=V_1'\cup V_2'$,
%we have $N_{G'}(w_1)\subseteq V_{i}'$, as desired.
%\end{proof}

In what follows, we divide the proof into the following three cases with respect to the values of $q$ and $r$.

\noindent{\textbf{Case 1.}} $r$ is even.

From Theorem \ref{THM1.2} we know that $K_{\lceil\frac{n-1}{2}\rceil,\lfloor\frac{n-1}{2}\rfloor}\bullet K_3$ is $\theta(1,q,r)$-free.
Then,
\begin{align}\label{EQU-8}
\rho(G)&\geq \rho(K_{\lceil\frac{n-1}{2}\rceil,\lfloor\frac{n-1}{2}\rfloor}\bullet K_3)
> \rho(K_{\lceil\frac{n-1}{2}\rceil,\lfloor\frac{n-1}{2}\rfloor})
= \sqrt{\Big\lfloor\frac{(n-1)^2}{4}\Big\rfloor}.
\end{align}

Suppose first that $g=5$.
Then, $G$ is $C_3$-free.
By Lemma \ref{lem3.1}, $\rho(G)\leq \rho(SK_{\lceil\frac{n-1}{2}\rceil, \lfloor\frac{n-1}{2}\rfloor})$.
On the other hand,
since $SK_{\lceil\frac{n-1}{2}\rceil, \lfloor\frac{n-1}{2}\rfloor}$ is a proper subgraph of $K_{\lceil\frac{n-1}{2}\rceil,\lfloor\frac{n-1}{2}\rfloor}\bullet K_3$,
it follows that $\rho(G)\geq\rho(K_{\lceil\frac{n-1}{2}\rceil,\lfloor\frac{n-1}{2}\rfloor}\bullet K_3)
>\rho(SK_{\lceil\frac{n-1}{2}\rceil, \lfloor\frac{n-1}{2}\rfloor})$, which gives a contradiction.
Thus, $g=3$ and $C=w_1w_2w_3w_1$.
Recall that $N_{G'}(v)=N_{G}(v)\cap V(G')$ and $d_{G'}(v)=|N_{G'}(v)|$ for any $v\in V(G)$.
Without loss of generality,
we may assume that $d_{G'}(w_3)\leq d_{G'}(w_i)$ for  any $i\in \{1,2\}$.

\begin{claim}\label{CLA3.4}
For every $i\in \{1,2\}$, we have $d_{G'}(w_i)\geq \frac{n}{12}$.
\end{claim}

\begin{proof}
Assume that $G^*=G-\{w_1,w_3\}$.
Clearly, $d_{G'}(w_i)= d_G(w_i)-2$ for each $i\in \{1,2,3\}$.
Recursively applying Lemma \ref{lem3.3},
we can obtain that
\begin{align*}
\rho^2(G)&\leq \rho^2(G-\{w_1\})+2d_{G'}(w_1)+4
     \leq   \rho^2(G^*)+2d_{G'}(w_1)+2d_{G'}(w_3)+8.
\end{align*}
By Lemma \ref{lem3.4}, we have $\rho(G^*)\leq \sqrt{\lfloor\frac{(n-2)^2}{4}\rfloor}$.
Combining these with \eqref{EQU-8}, we obtain
\begin{align*}
4d_{G'}(w_1)\geq 2d_{G'}(w_1)+2d_{G'}(w_3)\geq \rho^2(G)-\rho^2(G^*)-8\geq \frac{n}{3},
\end{align*}
which yields that $d_{G'}(w_1)\geq \frac{n}{12}$.
Similarly, $d_{G'}(w_2)\geq \frac{n}{12}$.
\end{proof}

By Claim \ref{CLA3.3}, we may assume without loss of generality that $N_{G'}(w_1)\subseteq V_{1}'$
and $N_{G'}(w_2)\subseteq V_{j_0}'$ for some $j_0\in \{1,2\}$.
Furthermore, by a similar discussion as in Claim \ref{CLA2.11},
we have  $N_{G'}(w_i)\cap N_{G'}(w_j)=\varnothing$ for any $\{i,j\}\subseteq \{1,2,3\}$.

%\begin{claim}\label{CLA3.5}
%If $d_{G'}(w_3)\leq 24\varepsilon^{\frac12} n$,
%then $x_{w_3}\leq x_{w_i}$ for $i\in \{1,2\}$.
%\end{claim}
%
%
%\begin{proof}
%Since $d_{G'}(w_3)\leq 24\varepsilon^{\frac12} n$, we have
%$$\rho(G)x_{w_3}=x_{w_1}+x_{w_2}+\sum_{v\in N_{G'}(w_3)}x_v\leq 2x_{u^*}+d_{G'}(w_3)x_{u^*}\leq 25\varepsilon^{\frac12} nx_{u^*}.$$
%Combining Claims \ref{CLA3.2}, \ref{CLA3.4} and \eqref{EQU-1}, we have
%$$\rho(G)x_{w_i}>\sum_{v\in N_{G'}(w_i)}x_v\geq \frac{n}{12}(1-4\varepsilon^{\frac12})x_{u^*}>25\varepsilon^{\frac12} nx_{u^*}>\rho(G)x_{w_3},$$
%as desired.
%\end{proof}

\noindent{\textbf{Subcase 1.1.}} $N_{G'}(w_i)\subseteq V_1'$ for all $i\in \{1,2,3\}.$

%\textcolor{blue}{
Assume that $\{j_1,j_2,j_3\}=\{1,2,3\}$, where $x_{w_{j_3}}=\max_{w\in V(C)}x_{w}$.
Note that $N_{G'}(w_{j_i})\subseteq V_1'$ for each $i\in \{1,2\}$.
Then by Claim \ref{CLA3.4}, we get
$\bigcup_{i=1}^{2}N_{G'}(w_{j_i})=\bigcup_{i=1}^{2}N_{V_1'}(w_{j_i})\neq \varnothing$.
This implies that $\sum_{i=1}^{2}\sum_{v\in N_{V_1'}(w_{j_i})}x_v>0$.
Set $$G^*=G-\bigcup_{i=1}^{2}\{vw_{j_i}~|~v\in N_{V_1'}(w_{j_i})\}
+\bigcup_{i=1}^{2}\{vw_{j_3}~|~v\in N_{V_1'}(w_{j_i})\}.$$
Consequently,
\begin{align*}
\rho(G^*)-\rho(G)
\geq \mathbf{x}^{\mathsf{T}}\big(A(G^*)-A(G)\big)\mathbf{x}
= 2\sum_{i=1}^{2}\Big((x_{w_{j_3}}-x_{w_{j_i}})\cdot\sum\limits_{v\in N_{V_1'}(w_{j_i})}x_v\Big)
\geq 0.
\end{align*}
Since $\mathbf{x}$ is a positive eigenvector of $G$,
we have $\rho(G) x_{w_{j_3}}=\sum_{v\in N_{G}(w_{j_3})}x_v$.
If $\rho(G^*)=\rho(G)$,
then $\mathbf{x}$ is also a positive eigenvector of $G^*$,
and so
\begin{align*}
\rho(G^*) x_{w_{j_3}}
=\sum_{v\in N_{G}(w_{j_3})}x_v+\sum_{i=1}^{2}\sum_{v\in N_{V_1'}(w_{j_i})}x_v
>\sum_{v\in N_{G}(w_{j_3})}x_v=\rho(G) x_{w_{j_3}},
\end{align*}
which contradicts $\rho(G^*)=\rho(G)$.
Thus, $\rho(G^*)>\rho(G)$.
However, this, together with the fact that $G^*$ is non-bipartite (as $C\subseteq G^*$) and $\theta(1,q,r)$-free,
 contradicts the choice of $G$.
%}

\noindent{\textbf{Subcase 1.2.}} $N_{V_2'}(w_j)\neq \varnothing$ for some $j\in \{2,3\}$.

%\textcolor{blue}{
Assume that $\{j_1,j_2,j_3\}=\{1,2,3\}$,
where $x_{w_{j_i}}=\max\{x_{w_j}~|~N_{V_i'}(w_j)\neq \varnothing, 1\leq j\leq 3\}$ for each $i\in \{1,2\}$.
By Claim \ref{CLA3.4} and $N_{G'}(w_1)\subseteq V_{1}'$,
we get $N_{V_1'}(w_1)\neq \varnothing$.
Thus, $x_{w_{j_1}}\geq x_{w_{1}}$.
Combining these with \eqref{EQU-1}, Claims \ref{CLA3.2} and \ref{CLA3.4}, we obtain
$$d_{G}(w_{j_1})x_{u^*}\geq \rho(G)x_{w_{j_1}}\geq \rho(G)x_{w_1}>\sum_{v\in N_{G'}(w_1)}x_v\geq \frac{n}{12}(1-4\varepsilon^{\frac12})x_{u^*}>24\varepsilon^{\frac12} nx_{u^*}.$$
So, $d_{G}(w_{j_1})\geq 24\varepsilon^{\frac12} n$.
By Claim  \ref{CLA3.3} and  $N_{V_1'}(w_1)\neq \varnothing$,  we have $N_{G'}(w_{j_1})\subseteq V_1'$.
This, together with $N_{V_2'}(w_{j_2})\neq \varnothing$,
implies that $j_1\neq j_2$.
Set
$$G^*
=G-\bigcup_{i=1}^{2}\{vw_{j_3}~|~v\in N_{V_i'}(w_{j_3})\}
   +\bigcup_{i=1}^{2}\{vw_{j_i}~|~v\in N_{V_i'}(w_{j_3})\}.$$
Consequently,
\begin{align}\label{EQU-9}
\rho(G^*)-\rho(G)
\geq \mathbf{x}^{\mathsf{T}}\big(A(G^*)-A(G)\big)\mathbf{x}
= 2\sum_{i=1}^{2}\Big((x_{w_{j_i}}-x_{w_{j_3}})\cdot\sum\limits_{v\in N_{V_i'}(w_{j_3})}x_v\Big)
\geq 0.
\end{align}
Clearly, $N_{G^*}(w_{j_3})=\{w_{j_1},w_{j_2}\}$,
and hence $G^{*}\subseteq K_{|V_1'|+1,|V_2'|+1}\bullet K_3$.
Combining these with Lemma \ref{lem3.2}, we get
\begin{align*}
\rho(G)\leq \rho(G^*)\leq \rho(K_{|V_1'|+1,|V_2'|+1}\bullet K_3)\leq \rho(K_{\lceil\frac{n-1}{2}\rceil,\lfloor\frac{n-1}{2}\rfloor}\bullet K_3).
\end{align*}
Note that  $K_{\lceil\frac{n-1}{2}\rceil,\lfloor\frac{n-1}{2}\rfloor}\bullet K_3$ is  non-bipartite and  $\theta(1,q,r)$-free.
Then by the choice of $G$, we get $\rho(G)\geq \rho(K_{\lceil\frac{n-1}{2}\rceil,\lfloor\frac{n-1}{2}\rfloor}\bullet K_3)$.
Thus,
\begin{align}\label{EQU-10}
\rho(G)= \rho(G^*)= \rho(K_{|V_1'|+1,|V_2'|+1}\bullet K_3)= \rho(K_{\lceil\frac{n-1}{2}\rceil,\lfloor\frac{n-1}{2}\rfloor}\bullet K_3).
\end{align}
Furthermore, again by Lemma \ref{lem3.2} and $G^{*}\subseteq K_{|V_1'|+1,|V_2'|+1}\bullet K_3$,
we get $G^*\cong K_{|V_1'|+1,|V_2'|+1}\bullet K_3\cong K_{\lceil\frac{n-1}{2}\rceil,\lfloor\frac{n-1}{2}\rfloor}\bullet K_3$.
To prove $G\cong K_{\lceil\frac{n-1}{2}\rceil,\lfloor\frac{n-1}{2}\rfloor}\bullet K_3$,
it suffices to show that $G\cong G^*$.
Otherwise, $N_{V_i'}(w_{j_3})\neq \varnothing$ for some $i\in \{1,2\}$.
Since $\mathbf{x}$ is a positive eigenvector of $G$,
we have
$$\rho(G) x_{w_{j_3}}=\sum_{v\in N_{G}(w_{j_3})}x_v>x_{w_{j_1}}+x_{w_{j_2}}.$$
In view of \eqref{EQU-10} and \eqref{EQU-9}, $\rho(G^*)=\rho(G)$
and $\mathbf{x}$ is also a positive eigenvector of $G^*$.
Thus,
\begin{align*}
\rho(G) x_{w_{j_3}}=\rho(G^*) x_{w_{j_3}}=x_{w_{j_1}}+x_{w_{j_2}}<\rho(G) x_{w_{j_3}},
\end{align*}
a contradiction.
Thus, $G\cong G^*\cong K_{\lceil\frac{n-1}{2}\rceil,\lfloor\frac{n-1}{2}\rfloor}\bullet K_3$, as desired.

\noindent{\textbf{Case 2.}} $r$ is odd and $q=2$.

Let $H=SK_{\lceil\frac{n-1}{2}\rceil,\lfloor\frac{n-1}{2}\rfloor}-\{v\}$, where $v$ is the vertex of degree two in $SK_{\lceil\frac{n-1}{2}\rceil,\lfloor\frac{n-1}{2}\rfloor}$.
Then, $e(H)=\lfloor\frac{n^2-2n-3}{4}\rfloor$, and
$\rho(H)\geq \frac{\mathbf{1}^{\mathsf{T}}A(H)\mathbf{1}}{\mathbf{1}^{\mathsf{T}}\mathbf{1}}
  =\frac{2e(H)}{n-1}
  >\frac{n-1.1}{2}.$
Clearly, $SK_{\lceil\frac{n-1}{2}\rceil,\lfloor\frac{n-1}{2}\rfloor}$ is $C_{1+q}$-free, and hence $\theta(1,q,r)$-free.
Consequently,
\begin{align}\label{EQU-11}
\rho(G)\geq \rho(SK_{\lceil\frac{n-1}{2}\rceil,\lfloor\frac{n-1}{2}\rfloor})>\rho(H)
>\frac{n-1.1}{2}.
\end{align}

Suppose first that $g=3$. Then $C=w_1w_2w_3w_1$.
Assume without loss of generality that $d_{G'}(w_3)=\min_{w\in V(C)}d_{G'}(w)$.

\begin{claim}\label{CLA3.5}
For every $i\in \{1,2\}$, we have $d_{G'}(w_i)\geq \frac{n}{12}$.
\end{claim}

\begin{proof}
Assume that $G^*=G-\{w_1,w_3\}$.
Clearly, $d_{G'}(w_i)= d_G(w_i)-2$ for each $i\in \{1,2,3\}$.
Recursively applying Lemma \ref{lem3.3},
we can obtain that
\begin{align*}
\rho^2(G)&\leq \rho^2(G-\{w_1\})+2d_{G'}(w_1)+4
     \leq   \rho^2(G^*)+2d_{G'}(w_1)+2d_{G'}(w_3)+8.
\end{align*}
By Lemma \ref{lem3.4}, we have $\rho(G^*)\leq \sqrt{\lfloor\frac{(n-2)^2}{4}\rfloor}$.
Combining these with \eqref{EQU-11}, we obtain
\begin{align*}
4d_{G'}(w_1)\geq 2d_{G'}(w_1)+2d_{G'}(w_3)\geq \rho^2(G)-\rho^2(G^*)-8\geq \frac{n}{3},
\end{align*}
which yields that $d_{G'}(w_1)\geq \frac{n}{12}$.
Similarly, $d_{G'}(w_2)\geq \frac{n}{12}$.
\end{proof}

By Claim \ref{CLA3.3}, we may assume without loss of generality that  $N_{G'}(w_1)\subseteq V_{1}'$.
By a similar discussion as in Claim \ref{CLA2.13} we get the following:\\
($i$) for any integer $i\in \{1,2,3\}$, we have $N_{G'}(w_i)\subseteq V_1'$;\\
($ii$) for any $\{i,j\}\subseteq \{1,2,3\}$, $N_{V_1'}(w_i)\cap N_{V_1'}(w_j)=\varnothing$.\\
Assume that $\{j_1,j_2,j_3\}=\{1,2,3\}$, where $x_{w_{j_3}}=\max_{w\in V(C)}x_{w}$.
From Claim \ref{CLA3.5}
we know that $\bigcup_{i=1}^{2}N_{V_1'}(w_{j_i})\neq \varnothing$.
Thus, $\sum_{i=1}^{2}\sum_{v\in N_{V_1'}(w_{j_i})}x_v>0$.
Set $$G^*=G-\bigcup_{i=1}^{2}\{vw_{j_i}~|~v\in N_{V_1'}(w_{j_i})\}
+\bigcup_{i=1}^{2}\{vw_{j_3}~|~v\in N_{V_1'}(w_{j_i})\}.$$
Consequently,
\begin{align*}
\rho(G^*)-\rho(G)
\geq \mathbf{x}^{\mathsf{T}}\big(A(G^*)-A(G)\big)\mathbf{x}
= 2\sum_{i=1}^{2}\Big((x_{w_{j_3}}-x_{w_{j_i}})\cdot\sum\limits_{v\in N_{V_1'}(w_{j_i})}x_v\Big)
\geq 0.
\end{align*}
Since $\mathbf{x}$ is a positive eigenvector of $G$,
we have $\rho(G) x_{w_{j_3}}=\sum_{v\in N_{G}(w_{j_3})}x_v$.
If $\rho(G^*)=\rho(G)$,
then $\mathbf{x}$ is also a positive eigenvector of $G^*$,
and so
\begin{align*}
\rho(G^*) x_{w_{j_3}}
=\sum_{v\in N_{G}(w_{j_3})}x_v+\sum_{i=1}^{2}\sum_{v\in N_{V_1'}(w_{j_i})}x_v
>\sum_{v\in N_{G}(w_{j_3})}x_v=\rho(G) x_{w_{j_3}},
\end{align*}
which contradicts $\rho(G^*)=\rho(G)$.
Thus, $\rho(G^*)>\rho(G)$.
However, this, together with the fact that $G^*$ is non-bipartite (as $C\subseteq G$) and $\theta(1,q,r)$-free,
contradicts the choice of $G$.

Thus $g\geq5$, which indicates that $G$ is $C_3$-free.
By Lemma \ref{lem3.1}, $\rho(G)\leq \rho(SK_{\lceil\frac{n-1}{2}\rceil, \lfloor\frac{n-1}{2}\rfloor})$ with equality if and only if $G\cong SK_{\lceil\frac{n-1}{2}\rceil, \lfloor\frac{n-1}{2}\rfloor}$.
Combining this with  \eqref{EQU-11} that $\rho(G)\geq \rho(SK_{\lceil\frac{n-1}{2}\rceil, \lfloor\frac{n-1}{2}\rfloor})$, we can find that
$G\cong SK_{\lceil\frac{n-1}{2}\rceil, \lfloor\frac{n-1}{2}\rfloor}$, as desired.

\noindent{\textbf{Case 3.}} $r$ is odd and $q\geq 4$.

We first give a claim.

\begin{claim}\label{CLA3.6}
For arbitrary $v_1,v_2,v_3\in V(C)$ with $d_{G'}(v_1)\geq d_{G'}(v_2)\geq d_{G'}(v_3)$,
$d_{G'}(v_1)\geq \frac{n}{18}$.
\end{claim}

\begin{proof}
Assume that $G^*=G-\{v_1,v_2,v_3\}$.
Clearly, $d_{G'}(v_i)= d_G(v_i)-2$ for each $i\in \{1,2,3\}$.
Recursively applying Lemma \ref{lem3.3},
we can obtain that
\begin{align*}
\rho^2(G)&\leq \rho^2(G-\{v_1\})+2d_{G'}(v_1)+4
     \leq   \rho^2(G-\{v_1,v_2\})+2d_{G'}(v_1)+2d_{G'}(v_2)+8\\
     &\leq \rho^2(G^*)+2d_{G'}(v_1)+2d_{G'}(v_2)+2d_{G'}(v_3)+12.
\end{align*}
Using Lemma \ref{lem3.4}, we have $\rho(G^*)\leq \sqrt{\lfloor\frac{(n-3)^2}{4}\rfloor}$.
Combining these with \eqref{EQU-5}, we obtain
\begin{align*}
6d_{G'}(v_1)\geq 2d_{G'}(v_1)+2d_{G'}(v_2)+2d_{G'}(v_3)\geq \rho^2(G)-\rho^2(G^*)-12\geq \frac{n}{3},
\end{align*}
which yields that $d_{G'}(v_1)\geq \frac{n}{18}$.
\end{proof}

Suppose first that $g=5$.
By Claim \ref{CLA3.6},
there exist at least three vertices $v\in V(C)$ satisfying $d_{G'}(v)\geq \frac{n}{18}$.
Among these three vertices, there exist two adjacent vertices, say $v_1$ and $v_2$, in the cycle $C$.
However, by a similar discussion as in Claim \ref{CLA2.14},
we have $d_{G'}(v_1)<\frac{n}{18}$ or $d_{G'}(v_2)<\frac{n}{18}$,
which gives a contradiction.

It remains the case $g=3$.
Recall that $C=w_1w_2w_3w_1$.
Assume without loss of generality that
$d_{G'}(w_3)=\max_{w\in V(C)}d_{G'}(w)$.
By Claim \ref{CLA3.6}, we have $d_{G'}(w_3)\geq \frac{n}{18}$.
By Claim \ref{CLA3.3}, we may assume without loss of generality that $N_{G'}(w_3)\subseteq V_{1}'$.
By a similar discussion as in Claim \ref{CLA2.16}, we have  $N_{G'}(w_i)=\varnothing$ for each $i\in \{1,2\}$.
Then, $G\subseteq K_{|V_1'|,|V_2'|+1}\circ K_3$.
Combining \eqref{EQU-5} and Lemma \ref{lem3.2} gives
$$\rho(K_{\lceil\frac{n-2}{2}\rceil,\lfloor\frac{n-2}{2}\rfloor}\circ K_3)\leq \rho(G)\leq \rho(K_{|V_1'|,|V_2'|+1}\circ K_3)\leq \rho(K_{\lceil\frac{n-2}{2}\rceil,\lfloor\frac{n-2}{2}\rfloor}\circ K_3),$$
and hence $G\cong K_{\lceil\frac{n-2}{2}\rceil,\lfloor\frac{n-2}{2}\rfloor}\circ K_3$, as desired.

This completes the proof of Theorem \ref{THM1.3}.
\end{proof}

\end{document}